\documentclass[12pt, twoside]{article}
\usepackage{amsmath,amsthm,amssymb}
\usepackage{times}
\usepackage{enumerate}
   
   \usepackage{url}

\theoremstyle{definition}
\newtheorem{thm}{Theorem}[section]
\newtheorem{cor}[thm]{Corollary}
\newtheorem{lem}[thm]{Lemma}
\newtheorem{defin}[thm]{Definition}

\newtheorem{mainthm}[thm]{Main Theorem}

\newtheorem*{xrem}{Remark}

\numberwithin{equation}{section}

\newcommand{\subjclass}[1]{\bigskip\noindent\emph{2010 Mathematics Subject Classification:}\enspace#1}
\newcommand{\keywords}[1]{\noindent\emph{Keywords:}\enspace#1}

\newtheorem{proposition}{Proposition}[section]
\newtheorem*{remarks}{Remarks}  

\DeclareMathOperator{\cof}{cof}

\newcommand{\dd}{\, \mathrm{d}}
\DeclareMathOperator{\Det}{Det}

\DeclareMathOperator{\dist}{dist}       
\renewcommand{\div}{\operatorname{div}}
\newcommand{\eps}{\varepsilon}

\DeclareMathOperator{\imT}{im_T}

\DeclareMathOperator{\loc}{loc}
\newcommand{\N}{\mathbb{N}}             

\newcommand{\R}{\mathbb{R}}             

\newcommand{\weakc}{\rightharpoonup}
\newcommand{\weakcs}{\overset{*}{\rightharpoonup}}

\begin{document}


\baselineskip=17pt


\title{A lower bound for the void coalescence load in nonlinearly elastic solids}

\author{Victor Ca\~nulef-Aguilar\\
Facultad de Matem\'aticas, Pontificia Universidad Cat\'olica de Chile\\ Vicu\~na Mackenna 4860, Macul, Santiago, Chile\\
vacanulef@uc.cl \bigskip \\ 
Duvan Henao\\
Facultad de Matem\'aticas, Pontificia Universidad Cat\'olica de Chile\\ Vicu\~na Mackenna 4860, Macul, Santiago, Chile\\
dhenao@mat.puc.cl}

\date{20 June 2019}

\maketitle

\begin{abstract}
The problem of the sudden growth and coalescence of voids in elastic media is considered.
  The Dirichlet energy is minimized among incompressible and invertible Sobolev deformations
  of a two-dimensional domain having $n$ microvoids of radius $\eps$.
  The constraint is added that the cavities should reach at least certain minimum areas $\upsilon_1, \ldots, \upsilon_n$
  after the deformation takes place. They can be thought of as the current areas of the cavities during a quasistatic loading, 
  the variational problem being the way to determine the state to be attained by the elastic body in a subsequent time step.
  It is proved that if each $\upsilon_i$ is smaller than the area of a disk having a certain well defined radius, which is comparable
  to the distance, in the reference configuration, to either the boundary of the domain or the nearest cavity (whichever is closer),
  then there exists a range of external loads 
  for which
  the cavities opened in the body are circular in the $\eps\to 0$ limit. 
  In light of the results by Sivalonagathan \& Spector and Henao \& Serfaty  
  that cavities always prefer to have a circular shape (unless prevented to do so by the constraint of incompressibility),
  our theorem suggests
  that the elongation and coalescence of the cavities experimentally and numerically observed for large loads 
  can only take place after all the cavities have attained a volume comparable to the space they have available in the reference configuration.
  Based on the previous work of Henao \& Serfaty, who apply 
  the Ginzburg-Landau theory for superconductivity to the cavitation problem, 
  this paper shows how the study of the interaction of the cavities 
  is connected to the following more basic question: for what cavitation sites
  $a_1,\ldots, a_n$ and areas $v_1, \ldots, v_n$ 
  does there exist an incompressible and invertible deformation 
  producing cavities of those areas originating from those points. 
  In order to use the incompressible flow of Dacorogna \& Moser to answer that question,
  it is necessary to study first how do the 
  elliptic regularity estimates for the Neumann problem in domains with circular holes
  depend on the domain geometry.  
  

\end{abstract}

\section{Introduction}

\subsection{Cavitation and spherical symmetry}
\label{se:physical_motivation}

Cavitation in solids is the sudden formation and expansion 
of cavities in their interior in response to large triaxial loads.
The first experimental studies in elastomers are due to 
Gent \& Lindley \cite{GeLi59}, who also theoretically estimated the hydrostatic load for rupture
by solving the non-linearised equilibrium equations for
an infinitely thick elastic shell under the assumption of radial symmetry.
The first analysis of the evolution of a cavity (beyond its nucleation) was due to Ball \cite{Ball82};
he showed that the one-parameter family of deformations
\begin{align} \label{eq:Ball-radial}
	u(x)=\sqrt[n]{|x|^n+L^n} \frac{x}{|x|},\quad L\geq 0,\quad n=2,3
\end{align}
constitutes a stable branch of weak solutions to the incompressible elasticity equations, 
which bifurcates from the homogeneous deformation at the dead-load
predicted by Gent \& Lindley.
The radial symmetry assumption, which persisted in this pioneering work,
was finally removed by M\"uller \& Spector \cite{MuSp95} and Sivaloganathan
\& Spector \cite{SiSp00}; they proved the existence of minimizers of
the elastic energy allowing for all sorts of cavitation configurations.
Lopez-Pamies, Idiart \& Nakamura \cite{LoIdNa11} and 
Negr\'on-Marrero \& Sivaloganathan \cite{NeSi12} discussed the onset of cavitation
under non-symmetric loadings. Mora-Corral \cite{Mora14} studied the quasistatic evolution of cavitation. 
We refer to \cite{Francfort18,Kumar18,Poulain18,Kumar18b}, the Introduction in \cite{HeSe13}, and the 
references therein for a more complete guide through the literature on this fracture mechanism.

The analyses \cite{SiSp10a,SiSp10b,HeSe13} and the numerical study \cite{LianLi11} suggest that 
the \emph{cavities inside an elastic body prefer to adopt a spherical shape} when pressurised 
by large and multiaxial external tensions, 
regardless of their shape and size at the onset of fracture (or in the rest state, if they  existed already). 
In particular, given any open $\mathcal B\subset \R^2$; any small $\eps>0$; 
any finite collection $a_1,\ldots,a_n\in \mathcal B$ of cavitation points;
and any incompressible 
deformation map $u: \mathcal B \setminus \bigcup_1^n \overline{B}_\eps(a_i) \to \R^2$
satisfying M\"uller \& Spector's invertibility condition (see Definition \ref{de:INV} and \cite{MuSp95});
using the arguments in \cite{HeSe13} it can be seen that
\begin{align} 
      \label{eq:ameliore}
 \int_{\mathcal B \setminus \bigcup_1^n \overline{B}_\eps(a_i)}
 \frac{|Du|^2 -1}{2} dx \geq 
 \sum_1^n v_i \log \frac{R}{\eps} + \sum_{i=1}^n v_i D_i^2 \log \frac{\min \{d_i, \sqrt{v_iD_i^2}\}}{\eps} - C
\end{align}
where $C$ is a universal constant
and
$R$, $d_i$, $v_i$, and $D_i$ respectively denote $R:=\dist(\{a_1, \ldots, a_n\}, \partial{\mathcal B})$;
the distance to the nearest cavitation point, or to $\partial \mathcal B$ should the outer boundary be closer to $a_i$; 
the area of a cavity coming from $B_\eps(a_i)$; and the Fraenkel asymmetry \cite{FuMaPr08} of the same cavity
(which measures how far is it from being a circle).
The first term on the right-hand side is the exact cost of a radially-symmetric cavitation; 
the prefactor of $|\log \eps|$ in the second term, on the other hand, 
is zero if and only if the cavities are circular.
This shows that it is very expensive to produce non-circular cavities, as stated above.
(A sketch of the proof can be found in Section \ref{se:distortion}.
The result is in 2D
but suggests that the same occurs in 3D elasticity.)

In spite of the previous energetic consideration,
\emph{if the external load is too large then an important
geometric obstruction frustrates the desire of producing only spherical cavities}. 
Although this is already explained in \cite{HeSe13}, let us briefly describe the situation. 
Consider again a body that is only two dimensional; that is furthermore a disk;
that is subject to the displacement condition 
$u(x) = \lambda x\ \forall\,x \in \partial B_{R_0}$,
for some $\lambda>1$ ($R_0$ being the domain radius); 
and that can open only two cavities. 
A necessary condition for circular cavities 
of areas $v_1$ and $v_2$ to be opened is that they be disjoint and
enclosed by the deformed outer boundary.
This is possible only when the sum 
$2\sqrt{\frac{v_1}{\pi}} + 2\sqrt{\frac{v_2}{\pi}}$ of their diameters is less than the outer diameter $2\lambda R_0$.
On the other hand,
if the body is incompressible (if none of its parts can change its area),
the areas occupied by the material after and before the deformation must coincide:
\begin{align}
      \label{eq:incompressibility_2cavs}
   \pi (\lambda R_0)^2 - (v_1+v_2) = \pi R_0^2 - O(\eps^2)
\end{align}
(the term of order $\eps^2$ accounts for the eventual preexisting microvoids). Hence,
the necessary condition reads 
\begin{align} 
      \label{eq:geom_obs}
  2\sqrt{v_1v_2}\leq \pi R_0^2 - O(\eps^2).
\end{align}
It follows, for instance,
that if $\lambda>\sqrt{2}$ then the body cannot open two circular cavities
of the same size.

The conflict between the geometric obstruction due to incompressibility and the energetic cost of distorted cavities 
raises the question of:\\
\centerline{\emph{What is the maximum load compatible with the opening of only spherical cavities?}}
In order to address this question, first we need to take the following into account.
It does not lead far to think of the load as just a scalar: it is more appropriate to consider the whole combination of the displacement condition 
at the outer boundary; the cavitation sites in the reference configuration; and the size that each cavity is expected to attain; as the load. Consider, 
for example, the following trivial observation:
Equations \eqref{eq:incompressibility_2cavs} and \eqref{eq:geom_obs} impose no limit on $\lambda$ if $v_1$ and $v_2$ are taken to be, respectively, 
as $(\lambda^2-1)\cdot \pi R_0^2 + O(\eps^2)$ and zero.
The obstruction arises when all the $n$ cavities grow from a size of order $\eps$ to a size of order 1,
which corresponds naturally to the situation in quasistatic and dynamic loadings
(once a cavity forms and grows it is not expected to shrink back;  
healing is possible, however, upon compression and/or unloading \cite{Francfort18,Kumar18,Poulain18,Kumar18b}).
\label{obstruction}

The pair $\Big ( (a_i)_{i=1}^n, (v_i)_{i=1}^n \Big )$,
composed of the cavitation sites in the reference configuration
and the areas that the cavities are expected to attain,
will be referred to as \emph{a cavitation configuration},
and will describe, as explained in the previous paragraph, 
what we will understand as the load being exerted on the incompressible body.
(The stretch factor $\lambda$ for the outer boundary is determined by the 
areas $v_i$, due to the incompressibility.)
There is a simple geometric condition that is necessary 
for a $2D$ incompressible body to be continuously transformed from its rest state
into a stretched state producing the given cavitation configuration, maintaining during the process
the circular shape of all the cavities:
\begin{align}
  \label{geometric_condition}
  \begin{gathered}
    \text{that an evolution of the domain exists in which the total enclosed area is preserved}
  \end{gathered}
\end{align}
(see Definition \ref{de:attainable}).
This paper's answer to the question of the previous paragraph is that 
\begin{center}
 \it As long as the external load fulfils the simple necessary
 geometric condition \eqref{geometric_condition},
 an incompressible and invertible deformation will always exist
 that opens round cavities
of the desired sizes at the desired sites. \end{center}
This is made explicit in Theorem \ref{th:round}.
For simplicity, the theorem is proved only in the case when
also the outer boundary of the domain is circular (not only the cavities).
However, it should be possible to extend all the analysis in this paper
to any smooth (or even Lipschitz) outer boundary.

\subsection{Void coalescence}
\label{se:Int_coalescence}

There is an extensive literature about the coalescence of voids in elastomers and in ductile materials.
On the experimental side, see, e.g., \cite{Gent91,Poulain17,PeCuSiEl06}.
On the numerical and modelling side, and restricting our attention, for concreteness, 
to the case of elastomers, 
see both
\cite{XuHe11,LianLi11,LianLi11JCPAM,LianLi12,Lefevre15},
which focus on the building-up of tension before coalescence
(only Sobolev maps are considered in the energy minimization),
and the SBV models \cite{HeMoXu16,Kumar18} 
(based on \cite{BoFrMa08}
and the analyses \cite{HeMo11,HeMo12,HeMo15,HeMoXu15}),
where the interaction can be followed all the way up to 
the nucleation and propagation of cracks.

What is observed during the quasistatic loading of a confined elastomer
is that 
cavities eventually lose their spherical shape 
as the load increases,
and begin to interact with other cavities until they merge into micro-cracks.
It follows that \emph{if for a certain load it is possible to prove that the cavities formed inside the body are close to spherical, 
then that load constitutes a lower bound for the load at which the voids begin to coalesce}.
For 2D neo-Hookean materials, such radial symmetry result can in fact be obtained, as shown by Henao \& Serfaty \cite{HeSe13},
using the methods and ideas developed for Ginzburg-Landau superconductivity. 
The existence question addressed in this paper, namely, that of determining 
for what loads there exists at least one deformation having finite energy and
opening only round cavities (regardless of whether it is energy minimizing), happens to play an important role in that
more complete radial symmetry statement. For those loads it can be shown that 
the cavities opened \underline{by the actual energy minimizers} are
also close to being circular. Consequently, 
\emph{by finding out a condition on the load sufficient to ensure that deformations with round cavities still exist 
(which is what we do in Theorem \ref{th:round}, as explained in the previous section),
we have, at the same time, obtained a lower bound for the coalescence load in the 2D neo-Hookean model.}
This is what lies behind Theorem \ref{th:main}.

Corollary \ref{co:illustrates} 
gives a sense of what is required of
a load (of a cavitation configuration)
in order to satisfy the geometric condition \eqref{geometric_condition} 
for the opening of only round cavities.
This motivates Theorem \ref{pr:example_attainable}, a modified version
of Theorem \ref{th:main} where a slightly more general (and more realistic)
variational problem is considered.
On the one hand, the theorem yields the more explicit lower bound
$$\lambda= \left ( 1- \max \frac{\sum_k \pi d_k^2}{\pi R_0^2}\right )^{-1/2}
$$
for the coalescence load in terms of the stretch at the outer boundary,
where $\pi R_0^2$ is the area of the initial domain $B_{R_0}(0)$ and the 
maximum is taken over all collections 
\begin{align*}
 B_{d_1}(a_1), B_{d_2}(a_2), \ldots, B_{d_n}(a_n)
\end{align*}
of disjoint disks, centered at the prescribed cavitation sites $a_1$, $a_2$, \ldots, $a_n$, contained in the initial domain.
On the other hand, the way in which the theorem follows from the construction used in the proof brings some evidence 
to the conjecture, implicitly present already in \cite{BaMu84,HeSe13}, 
that \emph{any given cavity will retain its spherical shape as long as its radius, after the deformation, remains smaller or comparable 
to the distance, in the undeformed configuration, to the nearest cavitation point
(or the outer boundary, if it is closer)},
and that \emph{no coalescence ought to take place until all the cavities have attained that critical size}.

\subsection{The flow of Dacorogna \& Moser}
\label{se:Int_DM}

As explained in the previous section, 
when a $2D$ incompressible elastic body $\mathcal B$ with initial voids $B_\eps(a_1)$, \ldots,
$B_\eps(a_n)$
is subject to a tensile radially symmetric stretch at the outer boundary, the free boundary problem
of determining the size and shape of the voids after the deformation is important 
in order to understand the nucleation of cracks in its interior. It was also mentioned that
for loads below a certain critical value the problem reduces
to the simpler problem discussed in Section \ref{se:physical_motivation}, namely,
solving the nonlinear equation of incompressibility
$$\det Du(x)=1 \quad\text{for a.e.\ }x\in {\mathcal B}_\eps:={\mathcal B}\setminus \bigcup_{i=1}^n \overline B_\eps(a_i)$$
under Dirichlet conditions of the form
\begin{gather*}
 u(x) = \lambda x,\quad x\in \partial \mathcal{B},\\
 u(a_i+\eps e^{i\theta})= \zeta_i + \sqrt{v_i/\pi + \eps^2} e^{i\theta},\quad i\in \{1,\ldots, n\},\quad \theta\in [0,2\pi],
\end{gather*}
where the centers $\zeta_i$ of the cavities after the deformation can be determined freely. 
As will be explained in Section \ref{se:coalescence},
\emph{it is important to prove the existence of not just any solution,
but of a family of solutions 
$\{\tilde u_\eps\}_{\eps>0}$ verifying the energy upper bound
\begin{align}
  \label{eq:Int_EUB}
  \int_{{\mathcal B}_\eps} \frac{|D\tilde u_\eps|^2}{2} dx\leq C 
  + \sum_{i=1}^n v_i |\log \eps|,
\end{align}
for some constant $C$ independent of $\eps$}.
This is achieved in Theorem \ref{th:round}.

To prove the above result,  following \cite{HeSe13},
\emph{we treat separately the regions adjacent to and far from the cavities}. A certain neighbourhood $B_{R_i}(a_i)$ is then 
assigned to each cavitation point $a_i$, inside which the maps $\tilde u_\eps$ are defined 
to be the unique incompressible and radially symmetric map
\begin{align*}
  a_i + re^{i\theta} \ \mapsto\ \zeta_i + \sqrt{v_i/\pi + r^2} e^{i\theta},\qquad \eps<r<R_i,\quad \theta\in [0,2\pi]
\end{align*}
expanding the $\eps$-cavity to an area of $v_i+\pi \eps^2$.
\emph{The main difficulty in the analysis in this paper 
is to harmoniously glue those radially symmetric cavitation maps
defined near the cavities
into a single deformation of the whole of $\mathcal B$ that continues to be 
injective and incompressible}, that is, 
to prove the existence of a bijection $u_{\text{far}}$ from $\mathcal B \setminus \bigcup_i \overline B_{R_i}(a_i)$
onto $(\lambda \mathcal B) \setminus \bigcup_i \overline B_{\sqrt{v_i/\pi+ R_i^2}}(\zeta_i)$
satisfying the Dirichlet conditions 
\begin{gather*}
 u_{\text{far}}(x) = \lambda x,\quad x\in \partial \mathcal{B},\\
 u_{\text{far}}(a_i+R_ie^{i\theta})= \zeta_i + \sqrt{v_i/\pi + R_i^2} e^{i\theta},\quad i\in \{1,\ldots, n\},\quad \theta\in [0,2\pi]
\end{gather*}
and the incompressibility constraint.
\emph{That such a bijection exists was proved for $n=2$ in \cite{HeSe13}; this paper extends that result to the case of an arbitrarily large number of cavities}
(Theorem \ref{th:scission-away}). It is by combining the lower bound 
and the ideas in \cite{HeSe13} with the upper bound \eqref{eq:Int_EUB}, obtained now for arbitrary $n$, that the main conclusions of this paper
(described in the previous sections)
are obtained
(see Theorems \ref{th:main} and \ref{pr:example_attainable}). 

The solution given in \cite{HeSe13} for $n=2$ was to construct explicit and carefully designed maps satisfying 
the above-mentioned requirements, but that technique 
is inapplicable in the presence of even just a third cavity. 
In order to explain our method of proof, recall first that
we restrict our attention to configurations  $\Big ( (a_i)_{i=1}^n, (v_i)_{i=1}^n \Big )$
satisfying the condition \eqref{geometric_condition} that a domain 
with cavities of areas $v_i$ originating at the points $a_i$
is attainable through an evolution of circular cavities. 
More precisely, there exists an evolution
$z_i(t)$ of the centers and an evolution $L_i(t)$ of the radii, with
$$1\leq t\leq \lambda,\quad z_i(1)=a_i, \quad L_i(1)=0,\quad \pi L_i(\lambda)^2 =v_i.$$
The final centers $\zeta_i$ will be chosen to be $\zeta_i:=z_i(\lambda)$.
This induces an evolution of the interface 
$$ \partial B_{z_i(t)}(r_i(t)), \quad r_i(t):=\sqrt{L_i(t)^2 + R_i^2}$$
between the region 
$$E(t):= (t\mathcal B) \setminus \bigcup_{i=1}^n B_{r_i(t)}(z_i(t))$$
far from the cavities and the region adjacent to cavity $i$
(see Section \ref{se:scission}).
\emph{To treat the case of many cavities here we solve the incompressibility equation using
the strategy of Dacorogna \& Moser \cite{DaMo90}}
consisting in:
\begin{itemize}
 \item finding first, for each $t$, a divergence-free velocity field $\hat v(y,t)$ defined for $y\in E(t)$;
 \item then defining $u_{\text{far}}(x)$ as the final point $f(x, \lambda)$ of the trajectory $f(x,t)$ obtained from the 
 flow equation 
 \begin{align}
  \label{eq:flow1}
    \frac{\partial f}{\partial t}(x,t) = \hat v(f(x,t),t), \quad 1\leq t\leq \lambda
 \end{align}
 with initial condition $f(x,1)=x$.
\end{itemize}

As will be explained in Section \ref{se:Schauder}, 
the proof that the resulting map $u_{\text{far}}$
has finite Dirichlet energy
(or, even more so, proving that $u_{\text{far}}$ is smooth),
as required in order to obtain the upper bound \eqref{eq:Int_EUB},
depends on establishing that the Lipschitz seminorm 
$\|D\hat v\|_{L^\infty(E(t))}$
is uniformly bounded with respect to the `time' parameter $t$.
The divergence-free velocity is obtained by solving two coupled Neumann problems for the Laplacian;
therefore, the problem reduces to understanding how do the constants in the elliptic regularity theory
for the Neumann problem
depend on the geometry of the domain.
\emph{The analysis ends once we confirm (in Theorem \ref{th:Schauder})
that the elliptic regularity constants do not blow up for families of domains $\{E(t)\}_t$ with circular holes
that remain sufficiently far from each other.}

\section{Notation and preliminaries}

\subsection{General notation}

We work in dimension two. The closure of a set is denoted by $\overline A$ and its boundary by $\partial A$. 
The open ball of radius $r>0$ centered at $x\in \R^2$ is denoted by $B_r(x)$.
The function $\dist$ indicates the distance from a point to a set, or between two sets.

Given a square matrix $A\in \R^{2\times 2}$, its determinant is denoted by $\det A$. The cofactor matrix 
$\cof A$ satisfies $(\det A)I=A(\cof A)^T$, where $I$ denotes de identity matrix.
If $A$ is invertible, its inverse is denoted by $A^{ -1}$. 
The inner (dot) product of vectors and of matrices will be denoted by $\cdot$. 
The Euclidean norm of a vector $x$ is denoted by $|x|$, and the associated matrix norm is also denoted by
$|\!\cdot\!|$.
Given $a,b\in \R^2$, the tensor product $a\otimes b$
is the $2\times 2$ matrix whose component $(i,j)$ is $a_ib_j$.

\subsection{Function spaces}

We fix a value of $\alpha\in (0,1)$ and work with the norms 
$\left\|f\right\|_{\infty}:=\sup|f(x)|$ and
\begin{align*}
 [f]_{0,\alpha}&:=\sup_{x\neq y}\frac{|f(x)-f(y)|}{|x-y|^{\alpha}}, 
 & 
 \left\|f\right\|_{0,\alpha} &:=\left\|f\right\|_{\infty}+[f]_{0,\alpha},
 \end{align*}
 \begin{align*}
 [f]_{1,\alpha}&:=\sup_{x\neq y}\frac{|Df(x)-Df(y)|}{|x-y|^{\alpha}},
 &
 \left\|f\right\|_{1,\alpha}&:=\left\|f\right\|_{\infty}+\left\|Df\right\|_{\infty}+[f]_{1,\alpha}.
\end{align*}
In the Neumann problems to be studied, the boundary data will be related to functions $g$ in:
$$C_{per}^{0,\alpha}:=\{ g\in C_{loc}^{0,\alpha}(\mathbb{R}): g \text{ is $2\pi$-periodic} \}.$$
The expression $u_{,\beta}$ stands for $\partial_{\beta} u=\frac{\partial u}{\partial x_\beta}$.

\subsection{Green's function}
\label{se:notation_Green}
The inversion of $x\in \R^2$ with respect to $B_R(0)$ is 
$x^{*}=\frac{R^2}{|x|^2}x.$ Set
$$\Phi(x) :=\frac{-1}{2\pi}\log(|x|), \quad
\phi^x(y):=\frac{1}{2\pi}ln(|y-x^{*}|)-\frac{|y|^2}{4\pi R^2},
\quad
G_{N}(x,y):=\Phi (y-x)-\phi^x(y).$$

\subsection{Poincar\'e's constant}
For any given open set $E\subset \R^2$, set
\begin{align}
  \label{def_C_P}
C_P(E):= \sup \left \{ \|\phi\|_{L^2(E)}: \phi \in H^1(E) \text{ s.t. } \|D\phi\|_{L^2(E)}=1
\ \text{and}\ \int_E \phi =0\right \}.  
\end{align}

\subsection{Topological image and condition INV}

We give a succint definition of the topological image (see \cite{HeSe13} for more details).
\begin{defin}
 \label{de:imT}
 Let $u\in W^{1,p}(\partial B_r(x), \R^2)$ for some $x\in \R^2$, $r>0$, and $p>1$. Then
 $$\imT(u, B_r(x)):=\{y\in \R^2: \deg (u, \partial B_r(x), y)\ne 0\}.$$
\end{defin}

Given $u\in W^{1,p}(E, \R^2)$ and $x\in E$, there is a set $R_x\subset (0,\infty)$,
 which coincides a.e.\ with $\{r>0: B_r(x)\subset E\}$, such that $u|_{\partial B_r(x)}\in W^{1,p}$
 and both $\deg (u, \partial B_r(x), \cdot )$ and $\imT(u, B_r(x))$ are well defined for all $r\in R_x$.

 \begin{defin}
 \label{de:INV}
 We say that $u$ satisfies condition INV if for every $x\in E$ and every $r\in R_x$
 \begin{enumerate}[(i)]
  \item $u(z)\in \imT(u, B_r(x))$ for a.e.\ $z\in B_r(x)\cap E$ and
  \item $u(z)\in \R^2 \setminus \imT(u,B_r(x))$ for a.e.\ $z\in E\setminus B_r(x)$.
 \end{enumerate}
\end{defin}
 
 If $u$ satisfies condition INV then $\{\imT(u,B_r(x)): r\in R_x\}$ is increasing in $r$ for every $x$.
 
 \begin{defin}
 Given $a\in E$ we define $$\imT(u,a):= \bigcap_{r\in R_a} \imT(u, B_r(a)).$$
 Analogously, if $u\in W^{i,p}$ is defined and satisfies condition INV in a domain of the form $E=\mathcal B
 \setminus \bigcup_1^n B_{r_i}(z_i)$, then 
 we define 
 $$\imT(u, B_{r_i}(z_i))= \bigcap_{\substack{r\in R_{z_i}\\ r>r_i}} \imT(u, B_r(z)).$$
 \end{defin}
 
\subsection{Distributional Jacobian}
\begin{defin}
 \label{de:DetDu}
 Given $u\in W^{1,2}(E,\R^2)\cap L^{\infty}_{\text{loc}}(E,\R^2)$ 
 its distributional Jacobian is defined as the distribution 
 $$ \langle \Det Du, \phi\rangle:= -\frac{1}{2}\int_E u(x)\cdot (\cof Du(x))D\phi(x)\dd x,\quad \phi\in C_c^{\infty}(E).$$
\end{defin}

\subsection{The cost of distortion}
\label{se:distortion}

We show how to adapt the proof of \cite[Prop.\ 1.1]{HeSe13} 
in order to obtain the refined estimate \eqref{eq:ameliore}
(which, as mentioned in the Introduction, shows clearly that round cavities 
are energetically preferred).

\begin{proof}[Proof of \eqref{eq:ameliore}]
Equations (3.3)-(3.4) in \cite{HeSe13} show that 
 \begin{align*}
  \int_{\mathcal B \setminus \bigcup \overline{B}_\eps(a_i)} \frac{|Du|^2-1}{2}dx
   \geq \sum_1^n v_i \log\frac{R/2}{n\eps}
     + C\int_{t_0}^{s_0} \left ( \sum_{B\in \mathcal B(t)} |E_B| D(E_B)^2 \right )\frac{dt}{t},
\end{align*}
where $t_0$, $s_0$, $\mathcal B(t)$ and $E_B$ are as in the proof of
\cite[Prop.\ 1.1]{HeSe13}
($E_B$ is an abbreviated notation for $\imT(u,B)$; 
it is the union of the the cavities opened from $B$ and of region occupied, in the deformed configuration, by the material points in $B$).
In the ball construction giving rise to the collection $\mathcal B(t)$, 
the radius $r(B)$ of every ball $B\in \mathcal B(t)$ is such that $r(B)\geq t/n$. 
Let $r_i$ be the radius of the largest among all the disks in the ball construction 
that are obtained as simple dilations of $B_\eps (a_i)$ (that is, before any merging 
process takes place). If $B_{d_i/2}(a_i)$ is disjoint with all balls 
in $\mathcal B(s_0)$, then there is no loss of generality in assuming that $r_i=\frac{d_i}{2}$.
If that were not the case, then $B_{r_i}(a_i)$ merges with other ball of $\mathcal B(t)$ precisely at `time' $t$.
Since $r\leq t$ holds true both for $r=r_i$ and for the radius of the other ball with which it merges, 
and since the other ball necesssarily contains other cavitation points, it follows that $d_i< 3t$. 
Since $r_i=r(B)$ for the ball $B=B_{r_i}(a_i)\in \mathcal B(t)$, by the observation at the beginning
of this paragraph we know that $r_i \geq t/n$. Therefore, $r_i \geq \frac{d_i}{3n}$. 
Taking this into account it can be seen that the above estimate can be replaced by 
$$
  \int_{\mathcal B \setminus \bigcup \overline{B}_\eps(a_i)} \frac{|Du|^2-1}{2}dx
   \geq \sum_1^n v_i \log\frac{R/2}{n\eps} + 
C\sum_{1}^n \int_{\eps}^{r_i} |E(a_i,r)| D(E(a_i, r))^2 \frac{dr}{r}. 
$$
By decreasing $r_i$, if necessary, it can be assumed that $r_i < \sqrt{\frac{v_iD_i^2}{24\pi}}$.
By virtue of \cite[Lemma 3.6.(ii)]{HeSe13}, this suffices to conclude that 
$|E(a_i,r)| D(E(a_i, r))^2\geq \frac{1}{2} v_iD_i^2$ for all $r\in (\eps, r_i)$. This completes the proof.
\end{proof}

\section{Attainable cavitation configurations}

As explained in the Introduction, given an initial domain radius $R_0$,
an arbitrary number $n$ of cavities,
cavitation sites $a_1, \ldots, a_n$,
target areas $v_1, \ldots, v_n$,
and a stretch factor $\lambda>1$,
we are interested on whether an incompressible body
occupying at rest the region $B_{R_0}(0)\setminus \bigcup_i \overline {B}_\eps(a_i)$
can be continuously stretched up to the point in which 
the outer boundary becomes $\partial B_{\lambda R_0}(0)$,
keeping the cavities always circular during the deformation.
A simple geometric necessary condition is that 
an evolution exists for the centers, the outer boundary radius, and
the cavity radii such that the total area enclosed by the intermediate
domains remains constant. This is expressed more precisely in the following
definition.

\begin{defin}
    \label{de:attainable}
    Let $n\in \N$, $R_0>0$, and ${\mathcal{B}}:=B_{R_0}(0)\subset \R^2$. We say that $\Big ( (a_i)_{i=1}^n, (v_i)_{i=1}^n \Big )$ 
    is a (cavitation) configuration attainable 
    through an evolution of circular cavities (or, more briefly, an \emph{attainable configuration})
    if $a_i \in{\mathcal{B}}$ and $v_i>0$ for all $i\in \{1,\ldots, n\}$,
    and  there exist evolutions
    \begin{itemize}
     \item $z_i\in C^1([1,\lambda], \R^2)$ of the cavity centers, and
     \item $L_i:[1,\lambda]\to [0,\infty)$ of the cavity radii,
    \end{itemize}
    where $\lambda$ is given by 
    \begin{align} \label{eq:defLambda}
\sum_{i=1}^n v_i = (\lambda^2-1)\pi R_0^2,
    \end{align}
    such that 
    \begin{align} \label{eq:incLi}
     \sum_{i=1}^n \pi L_i^2(t) = (t^2-1)\pi R_0^2\qquad \forall\, t\in [1,\lambda]
    \end{align}
     and for each $i\in \{1,\ldots, n\}$
    \begin{enumerate}[\upshape (i)]
    \item \label{it:Lsmooth}
      $L_i^2$ belongs to $C^1([1,\lambda],[0,\infty))$;
     \item $z_i(1)=a_i$ and $L_i(1)=0$;
     \item $\pi L_i^2(\lambda)=v_i$; and 
     \item \label{it:disjoint}
     for all $t\in [1,\lambda]$ the disks $\overline{B}_{L_i(t)}(z_i(t))$ are disjoint 
     and contained in $B_{tR_0}(0)$.
    \end{enumerate}
\end{defin}
  
Although other time parametrizations are of course possible for the evolution of the centers
and the radii in the above definition, we have chosen the stretch factor at the outer boundary
$\partial {\mathcal{B}}$ as our parameter.
Also, note that at the initial time $\lambda=1$ the radii $L_i(1)$ are being asked to be zero
instead of $\eps$; this is to ensure that the desired deformation
can be constructed for any small $\eps$.

\subsection*{Examples}

The following examples give a sense of 
what is required of a configuration $\Big ( (a_i)_{i=1}^n, (v_i)_{i=1}^n \Big )$
in order to be attainable through an evolution of circular cavities. We begin with a general result;
the more concrete examples are obtained as its 	corollaries.

\begin{lem} \label{le:sigma}
  Let $n\in \N$, $a_1,\ldots, a_n\in {\mathcal{B}}:=B_{R_0}(0)\subset \R^2$,  $v_1,\ldots, v_n>0$.
  Let $\lambda>1$ be such that $(\lambda^2-1)\pi R_0^2 = \sum v_i$. Set 
  \begin{align}
    \sigma = \min \left \{ \min_ i \frac{\left ( 1-\frac{|a_i|}{R_0}\right)^2}{\frac{v_i}{\sum v_k}},
    \min_{i\ne j} \frac{ |a_i-a_j|^2}{R_0^2 \left ( \sqrt{\frac{v_i}{\sum v_k}} + \sqrt{\frac{v_j}{\sum v_k}} \right )^2}
    \right \}.
  \end{align}
  Then both in the case $\sigma\geq 1$ and in the case
  when $\sigma<1$ 
  and $\lambda^2 < \frac{1}{1-\sigma}$
  the configuration $\Big ( (a_i)_{i=1}^n, (v_i)_{i=1}^n \Big )$
is attainable through an evolution of circular cavities.
\end{lem}

\begin{proof}
  For every $t\in [1,\lambda]$ and every $i\in \{1,\ldots, n\}$ set 
  \begin{align}
    z_i(t):=ta_i, \quad L_i(t):=\sqrt{(t^2-1)\frac{v_i}{\sum v_k}}\cdot R_0.
  \end{align}
  We only need to check that the $\overline{B_{L_i(t)}(z_i(t))}$ are disjoint and contained in $B_{tR_0}(0)$ for all $t$
  (the remaining conditions in Definition \ref{de:attainable} are immediately verified).
  Both in the case $\sigma\geq 1$ and in the case $\sigma<1$ and $\lambda^2 < \frac{1}{1-\sigma}$
  we have that 
  $$ 1-\lambda^{-2} < \sigma.$$
  As a consequence, we obtain that 
  $$1-t^{-2}<\sigma\quad \forall\,t\in[1,\lambda].$$
  Hence, $$1-t^{-2} < \frac{\left ( 1-\frac{|a_i|}{R_0}\right)^2}{\frac{v_i}{\sum v_k}}\quad \forall\,i $$
  and 
  $$1-t^{-2} < \frac{ |a_i-a_j|^2}{R_0^2 \left ( \sqrt{\frac{v_i}{\sum v_k}} + \sqrt{\frac{v_j}{\sum v_k}} \right )^2}
  \quad \forall\,i\ne j.$$
  It is easy to see that the first inequality is equivalent to $$ L_i(t)^2 < t^2(R_0-|a_i|)^2$$
  which in turn says that $L_i(t)+|z_i(t)| < tR_0$ (i.e., each $\overline{B_{L_i(t)}(z_i(t))}\subset B_{tR_0}(0)$).
  Analogously, the second inequality is equivalent to $$(\sqrt{L_i(t)}+\sqrt{L_j(t)})^2< t^2|a_i-a_j|^2$$
  which in turn says that $L_i(t)+L_j(t) < |z_i(t)-z_j(t)|$ (i.e., the disks are disjoint). 
  This completes the proof.
\end{proof}

  In the case when $v_1=v_2=\cdots=v_n$,
 \begin{align} \label{eq:packing}
   \sigma &= \frac{\displaystyle n\pi \min\left \{ \min_i (R_0-|a_i|)^2, \min_{i\ne j} \left ( \frac{|a_i-a_j|}{2}\right )^2 \right \}}{\pi R_0^2}.
  \end{align}
  This is the packing density of the largest disjoint collection of the form $\{ B_{\rho}(a_i): i\in\{1,\ldots,n\}\}$
  contained in ${\mathcal{B}}$ (same $\rho$ for all $i$). There is an extensive literature on the famous circle packing problem; 
  for example, it is known \cite{Melissen94} that when $n=11$ the maximum packing density is  
  $$ \frac{11}{\left (1+\frac{1}{\sin \frac{\pi}{9}}\right )^2} \approx 0.7145,$$
  which yields the upper bound 
  $$\lambda < \sqrt{\frac{(1+\sin \frac{\pi}{9})^2}{1+2\sin \frac{\pi}{9}- 10\sin^2\frac{\pi}{9}}}\approx 1.8714
  $$
  for which our above construction is able to produce attainable configurations with 11 cavities of equal size.

\begin{cor}
  Given any $n\in \N$, $a_1,\ldots, a_n\in {\mathcal{B}}:=B_{R_0}(0)\subset \R^2$,
  and $1\leq \lambda < \frac{1}{\sqrt{1-\sigma}}$, where $\sigma$
  is the maximum packing density \eqref{eq:packing},
  it is possible to attain the configuration of cavities of equal size compatible with the boundary condition $u(x)=\lambda x$ for $x\in \partial B_{R_0}(0)$
  (namely, $\Big ( (a_i)_{i=1}^n, (v_i)_{i=1}^n \Big )$ with
  $v_1=\cdots = v_n = \frac{\pi R_0^2}{n}(\lambda^2 -1)$).
\end{cor}

\begin{cor} \label{co:illustrates}
    Let $n\in \N$, $a_1,\ldots, a_n\in {\mathcal{B}}:=B_{R_0}(0)\subset \R^2$.
    If $d_1, \ldots, d_n>0$ are such that 
    the disks $\overline{B}_{d_i}(a_i)$ are disjoint and contained in $B_{R_0}(0)$,
    then the configuration $\Big ( (a_i)_{i=1}^n, (v_i)_{i=1}^n \Big )$, with 
    $$v_i=\pi d_i^2 \cdot \frac{1}{1-\frac{\sum \pi d_k^2}{\pi R_0^2}}\quad \forall i\in \{1,\ldots, n\},$$
    is attainable. 
\end{cor}

\begin{proof}
      We begin by noting that if $v_1,\ldots, v_n$
      are proportional to the areas of disks of radii $d_1,\ldots, d_n$ then 
      there is a simple sufficient condition for the hypothesis 
      $\lambda^2 < \frac{1}{1-\sigma}$ in Lemma \ref{le:sigma} to be satisfied. Indeed,
      suppose 
      \begin{align} \label{eq:proportional}
	\exists\, s>0\ \forall\, i \in \{1,\ldots, n\}\ v_i= \frac{s}{\displaystyle \sum_{k=1}^n \pi d_k^2} \pi d_i^2.
      \end{align}
      Then $\sigma > 1-\lambda^{-2}$ if and only if
      \begin{align*}
	    \forall i:\ \frac{\left (1-\frac{|a_i|}{R_0}\right )^2}{\frac{\pi d_i^2}{\sum \pi d_k^2}} > 1-\lambda^{-2}
	    \quad \text{and} \quad 
	    \forall i\ne j: \frac{|a_i-a_j|^2}{R_0^2 \frac{(d_i+d_j)^2}{\sum d_k^2}}> 1-\lambda^{-2}.
      \end{align*}
      This is equivalent to 
      \begin{align*}
       (1-\lambda^{-2})\frac{\pi R_0^2}{\sum_k \pi d_k^2} < \min \left \{ \min_i\ \left (\frac{R_0-|a_i|}{d_i}\right )^2,
       \min_{i\ne j} \frac{|a_i-a_j|^2}{(d_i+d_j)^2} \right \}. 
      \end{align*}
      The minimum on the right-hand side is greater than one because the disks $\overline{B}_{d_i}(a_i)$ are disjoint 
      and contained in $B_{R_0}(0)$. Hence, thanks to Lemma \ref{le:sigma},
      for $\Big ( (a_i)_{i=1}^n, (v_i)_{i=1}^n \Big )$ to be attainable it suffices  that 
      $(1-\lambda^{-2})  \leq \frac{\sum_k \pi d_k^2}{\pi R_0^2}$, i.e., 
      \begin{align} \label{eq:max_lambda}
	  \lambda^2 \leq  \frac{1}{1-\frac{\sum \pi d_k^2}{\pi R_0^2}}.
      \end{align}

      Recall that $\sum v_k = (\lambda^2-1)\pi R_0^2$,
      due to incompressibility. Since $\sum v_k=s$, the expression for $v_i$ in \eqref{eq:proportional}
      may be rewritten as 
      \begin{align} \label{eq:proportional2}
        v_i= (\lambda^2-1) \pi R_0^2 \frac{\pi d_i^2}{\sum \pi d_k^2}.
      \end{align} 
      The conclusion then follows by choosing the maximum value of $\lambda$ in \eqref{eq:max_lambda}.
\end{proof}

In the case of only one cavity, all loads are attainable, even if the cavitation point is close to the boundary.

\begin{proposition} \label{pr:n=1}
      All configurations with $n=1$ are attainable.
\end{proposition}

\begin{proof}
    Let $a\in B_{R_0}(0)$ and $\lambda>1$. We are to show that evolutions $t\in [1,\lambda] \mapsto z(t)$
    and $t\in [1,\lambda] \mapsto L(t)$ of the cavity's center and radius exist such that $z$ and $L^2$ are
    $C^1$, $z(1)=a$, $L(1)=0$, $\forall t: \pi L^2 (t)=(t^2-1)\pi R_0^2$, and $\forall t:\overline{B_{L(t)} (z(t))}\subset B_{tR_0}(0)$.
    It suffices to take $L(t):=\sqrt{t^2-1}R_0$ and $z(t):=(\lambda - \sqrt{\lambda^2-1}) a$, which are well defined
    actually for all $t\in [1,\infty)$.
\end{proof}

\section{Excision of holes off the domain}
\label{se:scission}

As mentioned in Section \ref{se:Int_DM},
in order to prove that for attainable load configurations
$\Big ( (a_i)_{i=1}^n, (v_i)_{i=1}^n \Big )$
there exist deformations maintaining the circular shape of the cavities 
(Theorem \ref{th:round}),
we proceed differently in the regions adjacent to the cavities
than in the connected region far from the cavities.
Near the cavities we work with the unique incompressible 
radially symmetric map that expands the $\eps$-cavity to an area
of $v_i+\pi \eps^2$. 
The analysis that follows is for the problem that remains, namely, 
for constructing an incompressible map far from the cavities.
The region near to cavity $i\in \{1,\ldots, n\}$ will end at a circumference centered
at $a_i$, the radius of which will be denoted by $R_i$.
The region far from the cavities:
$$E(1):= B_{R_0}(0)\setminus \bigcup_{i=1}^n \overline{B}_{R_i}(a_i),$$
will thus be 
a domain with holes that have ben cut from the original 
reference domain 
$${\mathcal B}_\eps= B_{R_0}(0)\setminus 
\bigcup_{i=1}^n \overline{B}_{\eps}(a_i).$$
The size of these holes is of order 1, as opposed to 
the initial cavities, which have radius $\eps$.

Let 
$$
    z_i:[1,\lambda]\to \R^2, \qquad L_i:[1,\lambda]\to [0,\infty), 
    \qquad i\in \{1,\ldots,n\}
$$
be the evolutions of the centers and of the cavity radii of 
Definition \ref{de:attainable}.
They induce the evolution 
\begin{align} \label{eq:domainsEt}
E(t) :=
B_{tR_0}(0)\setminus \bigcup_{i=1}^n \overline{B}_{r_i(t)} (z_i(t)),
\qquad 1\leq t\leq \lambda,
\end{align}
of the region far from 
the cavities, where the radii of holes evolve according to
\begin{align} \label{eq:defRiri}
 r_i(t):= \sqrt{ L_i(t)^2 + R_i^2}, \quad t\in [1,\lambda],\quad i\in \{1,\ldots, n\}.
\end{align}
The motivation for that expression is that:
\begin{itemize}
 \item  When $t=\lambda$
those are the radii that the holes must have
if
they are to coincide with the outer boundary of the image by 
the unique incompressible and radially symmetric map
defined on 
$$\{x: \eps<|x-a_i|<R_i\}$$
that enlarges the $\eps$-cavity to an area of $v_i+\pi\eps^2$.
\item The total area enclosed by $E(t)$ is always the same. Indeed, by virtue of \eqref{eq:incLi}, at all times $t$
$$
    |E(t)|= \pi (tR_0)^2 - \sum_{i=1}^n \pi r_i(t)^2 =
    \pi R_0^2 - \sum_{i=1}^n \pi R_i^2 = |E(1)|.
$$
\end{itemize}

We shall write 
$$z_0(t):=0,\qquad r_0(t):= tR_0, \qquad t\in [1,\lambda]$$
so that 
$$\partial E(t) = \bigcup_{i=0}^n \partial B_{r_i(t)}(z_i(t)).$$

So far nothing has been said regarding how the initial radii $R_i$ of the excised 
holes are to be chosen. 
They will only be required to be such that:
\begin{align} \label{eq:disjointEt}
    \text{at all times }
    t\in [1,\lambda],
    \text{ the holes }
    \overline{B}_{r_i(t)}(z_i(t))
    \text{ are disjoint and contained in }
    B_{tR_0}(0).
\end{align}
It is clear from
\eqref{eq:defRiri} and 
item (\ref{it:disjoint}) in 
Definition \ref{de:attainable}
that there exist radii 
$R_1, \ldots, R_n$
verifying that requirement
(by continuity).
As a consequence of \eqref{eq:disjointEt},
also using a continuity argument,
there exists $d>0$ such that 
\begin{align}
 \label{eq:disjoint_d}
 \begin{aligned}
 &\text{for all } t\in [1,\lambda]
 \text{ and all } i\in \{1,\ldots, n\}
 \ r_i(t)\geq d, \text{ and}
 \\
 &\text{for all } t\in [1,\lambda]
 \text{ the disks }\bigg \{\overline{B}_{r_i(t)+d}(z_i(t))\bigg \}_{i=1}^n
 \text{ are disjoint and contained in } B_{r_0(t)-d}(0).
\end{aligned}
\end{align}

Note finally that 
\begin{align} \label{eq:rmax}
  \text{for all } t\in [1,\lambda]
  \text{ and all } i\in \{1,\ldots, n\},
  \ r_i(t)\leq r_{max},
  \ \text{with } r_{max}:=\lambda R_0,
\end{align}
since
$$
  r_i(t)\leq r_0(t)=tR_0 < \lambda R_0.
$$

\section{Regularity of the Dacorogna-Moser velocity field}
\label{se:Schauder}

 \subsection{Definition of $v(y,t)$}

As mentioned in Section \ref{se:Int_DM},
the deformation  map $u$ away from the cavities, namely,
the part $u_{\text{far}}$ of $u$ that sends $E(1)$ onto $E(\lambda)$, 
is defined via the flow equation \eqref{eq:flow1}
using the velocity fields constructed 
 by Dacorogna \& Moser \cite{DaMo90}.
 In order to prove, in Theorem \ref{th:round}, that
 the resulting map lies in $H^1$, 
 we observe that 
 \begin{align} \label{eq:H1a}
   \frac{\dd}{\dd t} \int |D_xf(x,t)|^2\dd x
   &= \int D_xf(x,t)\cdot D_x \frac{\partial f}{\partial t} (x,t ) \dd x
   \\ &= \int D_xf(x,t)\cdot \Big (D_y\hat v(f(x,t),t)D_xf(x,t) \Big ),
  \end{align}
  whence 
  \begin{align} \label{eq:H1b}
    \frac{\dd}{\dd t} \int |D_xf(x,t)|^2\dd x \leq 
    \underbrace{(\sup_t \|D_y\hat v(\cdot, t)\|_{L^\infty(E(t))}  ) }_{:=C}
    \int |D_xf(x,t)|^2\dd x,
  \end{align}
  provided $\|D_y\hat v(\cdot, t)\|_{L^\infty(E(t))}$ is bounded with respect to $t$.
  This implies that $e^{-Ct} \int |D_xf(x,t)|^2$ decreases with $t$ and, consequently,
  $$\int _{E(1)}|Du_{\text{far}}|^2 dx\leq e^{C(\lambda-1)} \int |I|^2\dd x < \infty.$$
  Therefore, our aim is to prove that the spatial derivatives of the velocity field
  are bounded uniformly with respect to $t$.
  
  The field $\hat v$ will be obtained as the superposition of a field $v$
  that makes the excised holes grow with a field $\tilde v$ that translate
  the holes in order to change their centers to $z_i(t)$. 
  The first one will be defined as
  \begin{align}
    \label{eq:defDM}
    v(y,t):= D_y\phi_t(y) + D_y^\perp \psi_t(y),\quad t\in [1,\lambda], \quad y\in E(t),
  \end{align}
  where $\phi_t$ is the unique solution to
  \begin{align} \label{phi bvp}
\left\{ \begin{aligned}
&\Delta \phi_t=0\quad \text{in $ E(t)$,} \\
&\frac{\partial \phi_t}{\partial \nu}\Big (z_i(t) + r_i(t)e^{i\theta} \Big ) = 
\frac{\dd r_i(t) }{\dd t}
\quad \text{on $ \partial E(t)$}\\
& \int_{E(t)} \phi_t(y) dy =0
\end{aligned} \right. 
\end {align}
and $D^{\perp}\psi_t:=(\partial_{z_2}\psi,-\partial_{z_1}\psi)$ is a divergence-free covector field 
that cancels out the tangential parts of $D\phi_t$ on $\partial B_{r_i(t)}(z_i(t))$, for all $i\geq 1$. Concretely,
$$\psi_t(y):=\varphi_t(y)-\zeta\left(\frac{\dist(y,\partial E(t))}{d/2}\right)\varphi_t(q_t(y))$$ 
where $d$ satisfies \eqref{eq:disjoint_d},
$\varphi_t$ is the unique solution to
\begin{align} \label{varphi bvp}
\left\{ \begin{aligned}
&\Delta \varphi_t=0 \text{ in $ E(t)$,} \\
&\frac{\partial \varphi_t}{\partial \nu}=\frac{\partial \phi_t}{\partial \tau} \text{ on $ \partial E(t)$,}
\\
 &\int_{E(t)} \varphi_t(y) dy =0,
\end{aligned} \right. 
\end {align}
$\zeta$ is a cutoff function such that $0\leq \zeta  \leq 1$, $\zeta(0)=1$, and $\zeta(1)=0$,
and 
\begin{equation} \label{boundary proj}
q_t(z) :=
\left\{
	\begin{array}{ll}
		r_i(t)\frac{y-z_i(t)}{|y-z_i(t)|}+z_i(t)  & \mbox{if } |y-z_i(t)| <r_i(t)+\frac{d}{2} \\
		r_0(t)\frac{y}{|y|} & \mbox{if } |y| > r_0(t)-\frac{d}{2}.
	\end{array}
\right.
\end{equation}
The second field is defined as
\begin{align}
  \label{eq:defVtilde}
  \tilde v (y, t):= D_y^\perp w(y,t), \quad t\in [1,\lambda], \quad y\in \overline{E(t)},
\end{align}
with
$$ w(y,t):=\begin{cases}
            \eta \left ( \frac{r-r_i(t)}{d}\right ) \frac{\dd z_i(t)}{\dd t} \cdot (rie^{i\theta}),
            & \text{if } y=z_i(t) + re^{i\theta},\ r_i(t)\leq r < r_i(t) + d;
            \\
            0 & \text{in other case},
           \end{cases}
  $$
  the function $\eta$ being any $C_c^\infty([0,1])$ function such that $\eta(0)=1$
  and $\eta'(0)=0$. 
  
  In order to prove that $\|D_y \hat v(\cdot, t)\|_{L^\infty(E(t))}$ is bounded in time 
  we need to estimate 
  $$\|D^2_y\phi_t\|_{L^\infty(E(t))},
  \quad \|D^2_y\varphi_t\|_{L^\infty(E(t))},\quad \text{and}\quad 
  \|D^2_y w\|_{L^\infty(E(t))}.$$
  To this aim,
  in the rest of Section \ref{se:Schauder}
  we consider a generic domain with holes $E$ of the form 
  \begin{align} \label{eq:genericE}
    E:=B_{r_0}(z_0) \setminus \bigcup_{i=1}^n \overline{B}_{r_i}(z_i) \subset \R^2
  \end{align}
  satisfying, for some positive $d$ and $r_{max}$, that
\begin{align}
 \label{eq:d}
 \begin{aligned}
 &\text{for all }  i\in \{1,\ldots, n\}
 \ d\leq r_i\leq r_{max}, \text{ and}
 \\
 &
 \text{the disks }\bigg \{\overline{B}_{r_i+d}(z_i)\bigg \}_{i=1}^n
 \text{ are disjoint and contained in } B_{r_0-d}(z_0).
\end{aligned}
\end{align}
Also, with an abuse of notation, 
we study the generic Neumann problem 
\begin{align} \label{eq:defNeumann}
\begin{cases}
\Delta u =0 \text{ in } E,\\[0.5em]
\displaystyle \frac{\partial u}{\partial \nu}=g
 \text{ on } 
\partial E,
\\[0.5em]
\displaystyle \int_{E}u(y)dy=0,
\end{cases}
\end{align}
where the Neumann datum $g$ verifies the compatibility requirement:
\begin{equation} 
\int_{\partial B_{r_0}(z_0)}g=\sum_{k=1}^{n}\int_{\partial B_{r_k}(z_k)}g, \label{nec}
\end{equation}
with a view towards estimating $\|D^2u\|_{L^\infty(E)}$
and applying this result first to 
$$E=E(t),\quad u=\phi_t, \quad g|_{\partial B_{r_i}(z_i)}= 
\frac{\dd r_i(t) }{\dd t},$$
then to
$$E=E(t),\quad u=\varphi_t, \quad g = 
\frac{\dd \phi_t }{\dd \tau}.$$
Note that in the first case \eqref{nec} holds thanks to the 
conservation of area \eqref{eq:incLi}
and to \eqref{eq:defRiri}. In the second case it holds because the tangential derivative 
of $\phi_t$ integrates up to zero.

Since estimating the $L^\infty$ norm of the second derivatives constitutes a borderline problem,
we estimate instead $\|D^2u\|_{0,\alpha}$ for some $\alpha\in (0,1)$.
This Section \ref{se:Schauder} is devoted to showing that the elliptic regularity constant $C=C(E)$
in $$\|D^2u\|_{0,\alpha}\leq C \|g\|_{1,\alpha}$$
does not blow up as long as the holes $B_{r_i}(z_i)$
defining the domain $E$ remain far from each other and do not shrink down to zero radius. 

  \subsection{Dependence on the geometry of Poincar\'e's constant}
  
Recall the definition of $C_P(E)$ in \eqref{def_C_P}.
\begin{thm}   \label {thm 2}
Let $n\in\mathbb{N}$ and $0<\delta<1$.
There exists a universal constant $C(\delta)$ 
such that
$$C_P(E)\leq C(\delta)r_0$$
for $E=B_{r_0}(z_0)\setminus\bigcup_{i=1}^{n}B_{r_i}(z_i)$,
whenever $z_0,...z_n\in\mathbb{R}^2$ and $d,r_0,...,r_n>0$
satisfy $\frac{d}{r_0}\geq \delta$ and \eqref{eq:d}.
\end{thm}

\begin{xrem}
 Note that $n\leq \delta^{-2}$, because \eqref{eq:d} implies 
 $$\bigcup_{i=1}^{n}B_{r_i+r_0\delta}(z_i)\subset B_{r_0}(z_0),$$ 
 which yields $$n(r_0\delta)^2\leq\sum_{i=1}^{n}(r_i+r_0\delta)^2\leq r_0^2.$$
\end{xrem}

\begin{proof}
\emph{It is enough to consider the case when $r_0=1$ and $z_0=0$.} Indeed,
suppose there exists such a constant $C(\delta)$ for domains with outer radius equal to $1$. Now consider a general $E$ (with $r_0$ not necessarily equal to $1$). Let $\phi\in H^{1}(E)$ be such that $\Vert D\phi\Vert_{L^{2}(E)}=1$ and $\int_{E}\phi=0$. Set $\psi(\omega):=\phi(z_0+r_0\omega)$, $\omega\in \hat{E}$ where $\hat{E}=\frac{E-z_0}{r_0}$. We have that
$$\int_{\hat{E}}|D\psi|^2=r_0^2\int_{\hat{E}}|D\phi(z_0+r_0\omega)|^2d\omega=\int_{E}|D\phi|^2dy=1$$
Clearly, we also have that $\int_{\hat{E}}\psi=0$. Then, by assumption, $\Vert \psi\Vert_{L^{2}}\leq C(\delta)$ (note that if $E$ satisfies the conditions in the statement then also does $\hat{E}$).
Hence $$\int_{E}\phi^2(y)dy=\int_{\hat{E}}\phi^2(z_0+r_0\omega)r_0^2d\omega=r_0^2\Vert\psi\Vert_{L^2(\hat{E})}^{2}\leq C(\delta)^2r_0^2;$$ since $\phi$ is arbitrary, this yields $C_P(E)\leq C(\delta)r_0$.
\medskip

\emph{Proof in the case $r_0=1$, $z_0=0$:}
 looking for a contradiction, suppose there exist $0<\delta<1$, a sequence of domains $(E_j)_{j\in \mathbb{N}}$ with unit outer radius and a sequence $(\phi_j)_{j\in\mathbb{N}}$ such that for all $j$:
 \begin{enumerate}[\upshape (i)]
\item $\phi_j\in H^1(E_j)$.
\item $\Vert D\phi_j\Vert_{L^2(E_j)}<\frac{1}{j}$, $\Vert\phi_j\Vert_{L^2(E_j)}=1$.
\item $\int_{E_j}\phi_j=0$.
\item Each $E_j$ satisfies the conditions in the statement of the theorem.
\end{enumerate}

Let $\tilde{\phi_j}$ denote the extensions of $\phi_j$ to $B_1(0)$. We have that 
$$\Vert \tilde{\phi_j} \Vert_{L^2(E_j)}\leq 2 \Vert \phi_j \Vert_{L^2(E_j)}=2
\quad \text{and}\quad 
\int_{B_1(0)}|D\tilde{\phi_j}|^2\leq C\left(\delta^{-2}+\frac{1}{j}\right).$$
Taking a subsequence, we obtain that $\tilde{\phi_j}\overset{H^1}{\rightharpoonup} \phi$ for some $\phi\in H^1(B_1(0))$. Also, a subsequence can be taken such that the centers $z_i^{(j)}$ and the radii $r_{i}^{(j)}$ of the holes of $E_j$ converge. Set $E$ be the limit domain. Clearly $|E\Delta E_j|\rightarrow 0$.

For every $E'= B_1(0)\setminus \bigcup_{i=1}^{i=n}B_{r_i'}(z_i)$ such that $E'\subset\subset E$ and such that the disks $\overline{B_{r_i'+\delta/2}(z_i)}$ are disjoint and contained in $B_{1-\delta/2}(0)$, we have that $D\tilde{\phi_j}=D\phi_j\rightarrow 0$ in $L^2(E')$ (because $\Vert D\phi_j \Vert_{L^2(E')} \leq \Vert D\phi_j \Vert_{L^2(E_j)} <\frac{1}{j}$ since $E'\subset E_j$ for sufficiently large $j$).
By uniqueness of weak limits, $D\phi\equiv 0$ in every such $E'$. Indeeed, for every $\eta\in C_{c}^{\infty}(E')$ we have that
$$\left|\int_{E'}\eta\partial_{\alpha}\phi\right|= \left|\int_{B_1(0)}\eta\partial_{\alpha}\phi\right|=\left|\lim_{j\rightarrow \infty}\int_{B_1(0)}\eta\partial_{\alpha}\tilde{\phi_j}\right|=\lim_{j\rightarrow \infty}\left|\int_{B_1(0)}\eta\partial_{\alpha}\tilde{\phi_j}\right|$$
$$=\lim_{j\rightarrow \infty}\left|\int_{E'}\eta\partial_{\alpha}\tilde{\phi_j}\right|\leq \limsup_{j\rightarrow \infty}\Vert \eta \Vert_{L^2(E')} \Vert D\tilde{\phi_j} \Vert_{L^2(E')}=0.$$
By the fundamental theorem of the calculus of variations, $\partial_{\alpha}\phi=0$ in $E'$.
It follows that $\phi|_{E'}$ is constant for every such $E'$. If $E',E''$ are two such domains and $E'\subset E''$, clearly the constant value of $\phi|_{E'}$ must coincide with the constant value of $\phi|_{E''}$, hence $\phi$ is constant in $E$.

Since $H^1(B_1(0))\subset \subset L^q(B_1(0))$ we can assume that for some $q>2$ $\tilde{\phi_j}\rightarrow\phi$ strongly in $L^q$. Thus
$$1=\lim_{j\rightarrow \infty}\int_{E_j}\phi_j^2=\lim_{j\rightarrow\infty}\int_{B_1(0)}\tilde{\phi_j}^2\chi_{E_j}=\int_{B_1(0)}\phi^2\chi_{E}$$
($ \tilde{\phi_j}^2\rightarrow \phi^2$ in $L^{\frac{q}{2}}$ and $\chi_{E_j}\rightarrow \chi_{E}$ in $L^{\left(\frac{q}{2}\right)'}$).
Analogously, 
$$0=\lim_{j\rightarrow \infty}\int_{E_j}\phi_j=\lim_{j\rightarrow\infty}\int_{B_1(0)}\tilde{\phi_j}\chi_{E_j}=\int_{B_1(0)}\phi\chi_{E}$$
Hence $\phi=0$ in $E$ (because $\phi$ was constant), but this contradicts that $\int_{B_1(0)}\phi^2\chi_{E}=1$. This completes the proof.
\end{proof}

\subsection{Dependence of the geometry of the trace constants}

\begin{lem} \label{trace}
Let $\phi\in H^{1}(B_{\rho_2}\setminus \overline{B_{\rho_1}})$ for some $0<\rho_1<\rho_2$. Then (for $i=1,2$):
$$\int_{\partial B_{\rho_i}}\phi^2(x)dS(x)\leq   
\frac{8}{\rho_2-\rho_1}\int_{B_{\rho_2}\setminus \overline{B_{\rho_1}}}\phi^2(x)dx+4(\rho_2-\rho_1)\int_{B_{\rho_2}\setminus \overline{B_{\rho_1}}}|D\phi|^2(x)dx  $$
\end{lem}

\begin{proof} 
 i) First we estimate $\int_{\partial B_{\rho_1}} \phi^2 dS$.
Given $\varepsilon >0$, let $\eta\in C^{\infty}(\overline{B_{\rho_2}}\setminus B_{\rho_1})$ be such that $\eta= 0$ on $\partial B_{\rho_2}$, $\eta=1$ on $\partial B_{\rho_1}$ and $|D\eta|\leq \frac{1+\varepsilon}{\rho_2-\rho_1}$.
$$    \int_{\partial B_{\rho_1}}\phi^2(x)dS(x)=\rho_1\int_{S^{1}}\left(\int_{\rho_1}^{\rho_2}\frac{d}{ds}((\eta \phi)(sz))ds\right)^2dS(z)     $$
$$\leq 2\rho_1(\rho_2-\rho_1)\int_{S^1}\int_{\rho_1}^{\rho_2}(|\phi D \eta|^2+|\eta D\phi|^2)dsdS(z)    $$
$$\leq 2\left(     \frac{(1+\varepsilon)^2}{\rho_2-\rho_1}\int_{\rho_1}^{\rho_2}\int_{S^{1}}\phi^2(sz)sdS(z)ds 
+(\rho_2-\rho_1)\int_{\rho_1}^{\rho_2}\int_{S^{1}}|D\phi|^2(sz)sdS(z)ds     \right).$$

\noindent ii) To estimate $\int_{\partial B_{\rho_2}} \phi^2 dS$, 
we consider first the case in which $\rho_1\geq\frac{\rho_2}{2}$: 
given $\varepsilon >0$, let $\eta\in C^{\infty}(\overline{B_{\rho_2}}\setminus B_{\rho_1})$ be such that $\eta= 1$ on $\partial B_{\rho_2}$, $\eta=0$ on $\partial B_{\rho_1}$ and $|D\eta|\leq \frac{1+\varepsilon}{\rho_2-\rho_1}$.
$$    \int_{\partial B_{\rho_2}}\phi^2(x)dS(x)=\rho_2\int_{S^{1}}\left(\int_{\rho_1}^{\rho_2}\frac{d}{ds}((\eta \phi)(sz))ds\right)^2dS(z)     $$
$$\leq 2\rho_2(\rho_2-\rho_1)\int_{S^1}\int_{\rho_1}^{\rho_2}(|\phi D \eta|^2+|\eta D\phi|^2)dsdS(z)    $$
$$\leq 4\left(     \frac{(1+\varepsilon)^2}{\rho_2-\rho_1}\int_{\rho_1}^{\rho_2}\int_{S^{1}}\phi^2(x)sdS(x)ds +(\rho_2-\rho_1)\int_{\rho_1}^{\rho_2}\int_{S^{1}}|D\phi|^2(x)sdS(x)ds     \right).$$

 Case in which $\rho_1<\frac{\rho_2}{2}$: by the previously considered case, 
since $H^{1}(B_{\rho_2}\setminus\overline{B_{\rho_1}})\subset H^{1}(B_{\rho_2}\setminus\overline{B_{\frac{\rho_2}{2}}})$ 
we have that
$$\int_{\partial \rho_2}\phi^2dx
\leq \frac{4}{\rho_2-\frac{\rho_2}{2}}\int_{B_{\rho_2}\setminus B_{\frac{\rho_2}{2}}}\phi^2 dx
+4\left(\rho_2-\frac{\rho_2}{2}\right)\int_{ B_{\rho_2} \setminus B_{\frac{\rho_2}{2}} } |D\phi|^2dx$$
$$\leq \frac{8}{\rho_2-\rho_1}\int_{B_{\rho_2}\setminus \overline{B_{\rho_1}}}\phi^2(x)dx
  +4(\rho_2-\rho_1)\int_{B_{\rho_2}\setminus \overline{B_{\rho_1}}}|D\phi|^2(x)dx. $$
\end{proof}

\subsection{Estimates in the interior of the domain}
\label{se:interior}

In the next subsection we shall focus on studying the regularity  near the boundary of $E$
of the solution to the Neumann problem \eqref{eq:defNeumann}. The results will be of the form:
`the $L^\infty$ and H\"older norms of $u$ and its derivatives in the annulus $r_i\leq |x-z_i|
< r_i + \frac{d}{3}$
are controlled by $u$ and its derivatives in the annulus $r_i+\frac{d}{3} < |x-z_i|< r_i + \frac{2d}{3}$'.
In this subsection we obtain estimates in this second annulus that lies a distance $\frac{d}{3}$ apart
from the boundary of $E$.
Since the analysis will be carried out separately around each hole $B_{r_i}(z_i)$, the localization
being possible by the multiplication with suitable cut-off functions,
we work in a generic annulus 
\begin{align}
 \label{eq:Omega}
 \Omega:=\{x\in \mathbb{R}^2: R<|x|<R+d\}.
\end{align}
For calculations that have to be made away from $\partial \Omega$,
we work in 
\begin{align} \label{eq:Omega'}
 \Omega':=\{x\in \mathbb{R}^2: R+\frac{1}{3}d<|x|<R+\frac{2}{3}d\}.
\end{align}
The generic radius $R$ in \eqref{eq:Omega} corresponds  
ultimately to the radius $r_i(t)$ of one of the holes of $E(t)$, for a fixed given $t$.
Recall that the role of the length $d$ is that of giving a uniform-in-time lower bound
for the width of an annular neighbourhood
of the excised hole that is still contained in $E(t)$,
as well as the lower bound $R\geq d$ for the radii (see \eqref{eq:d}).

The following regularity estimates for harmonic functions can be found in \cite[Thm.\ 2.2.7]{Evans10}
\begin{lem}   \label{harm reg}
Let $v$ be harmonic in $B_d(x)$, then:\\
$ \left\|v\right\|_{L^{\infty}(B_{d/2}(x))}\leq C d^{-2}\left\|v\right\|_{L^1(B_d(x))} .$\\
$ \left\|D^{\beta}v\right\|_{L^{\infty}(B_{d/2}(x))}\leq C d^{-2-|\beta|}\left\|v\right\|_{L^1(B_d(x))} .$
\end{lem}

\begin{proposition} \label{prop2}
For every positive $d$ and $r_{max}$, with $d<r_{max}$, there
exists a constant $C(d, r_{max})$ such that 
if $d\leq R\leq r_{max}$ and $v$ is harmonic in $\Omega$ then
$$ \left\|v\right\|_{L^{\infty}(\Omega')}
+ [v  ]_{0,\alpha(\Omega')}
+ \left\|Dv\right\|_{L^{\infty}(\Omega')}
+ [ v ]_{1,\alpha(\Omega')}
\leq C(d,r_{max})\left\|v\right\|_{L^1(\Omega)}. $$
\end {proposition}

\begin{proof}
 The estimates for the $L^\infty$ norm of $v$ and $Dv$ follow from the previous lemma.
To prove the estimate for $[v  ]_{0,\alpha}$ 
note that using polar coordinates we get (for $r\in (R+\frac{1}{3}d,R+\frac{2}{3}d)$ and $\theta_1,\theta_2\in [-\pi,\pi]$, such that $|\theta_1-\theta_2|\leq \pi$):
$$ |v(re^{i\theta_1})-v(re^{i\theta_2})|\leq \int_{\theta_1}^{\theta_2}\left|\frac{d}{d\theta}\left(v(re^{i\theta})\right)\right|d\theta \leq \int_{\theta_1}^{\theta_2}\left|\frac{\partial v}{\partial x_1}\right|r|\sin(\theta)|+\left|\frac{\partial v}{\partial x_2}\right|r|\cos(\theta)|d\theta$$
$$\leq Cd^{-3}\left\|v\right\|_{L^{1}(\Omega)}r|\theta_1-\theta_2|\leq Cd^{-3}\left\|v\right\|_{L^{1}(\Omega)}|re^{i\theta_1}-re^{i\theta_2}|^{\alpha}R^{1-\alpha},$$
 since $r|\theta_1-\theta_2|\leq \frac{\pi}{2}|re^{i\theta_1}-re^{i\theta_2}|$ (recall that $\frac{2}{\pi^2}\leq\frac{1-\cos(\theta)}{\theta^2} \leq\frac{1}{2}$, for $\theta\in [-\pi,\pi] $) and $|re^{i\theta_1}-re^{i\theta_2}|\leq 2r\leq CR$ .\\
Moreover, for $\theta\in [-\pi,\pi]$ and $r_1,r_2\in[R+\frac{1}{3}d,R+\frac{2}{3}d]$, we have:
$$|v(r_1e^{i\theta})-v(r_2e^{i\theta})|\leq \int_{r_1}^{r_2}\left|\frac{d}{dr}\left(v(re^{i\theta})\right)\right|dr\leq \int_{r_1}^{r_2}\left|\frac{\partial v}{\partial x_1}\right||\cos(\theta)|+\left|\frac{\partial v}{\partial x_2}\right||\sin(\theta)|dr$$
$$\leq Cd^{-3}\left\|v\right\|_{L^{1}(\Omega)}|r_1-r_2|\leq Cd^{-3}\left\|v\right\|_{L^{1}(\Omega)}|r_1e^{i\theta}-r_2e^{i\theta}|^{\alpha}R^{1-\alpha}$$

Now, for $r_1,r_2\in[R+\frac{1}{3}d,R+\frac{2}{3}d]$, $r_1\leq r_2$ and $\theta_1,\theta_2\in [-\pi,\pi]$, such that $|\theta_1-\theta_2|\leq \pi$, we have:
$$ |v(r_1e^{i\theta_1})-v(r_2e^{i\theta_2})|\leq |v(r_1e^{i\theta_1})-v(r_1e^{i\theta_2})|+|v(r_1e^{i\theta_2})-v(r_2e^{i\theta_2})|$$
$$\leq Cd^{-3}R^{1-\alpha}\left\|v\right\|_{L^{1}(\Omega)}(|r_1e^{i\theta_1}-r_1e^{i\theta_2}|^{\alpha}+|r_1e^{i\theta_2}-r_2e^{i\theta_2}|^{\alpha})$$
$$\leq Cd^{-3}R^{1-\alpha}\left\|v\right\|_{L^{1}(\Omega)}(|r_1e^{i\theta_1}-r_2e^{i\theta_2}|^{\alpha}+|r_1e^{i\theta_1}-r_2e^{i\theta_2}|^{\alpha}),$$
 since $|r_1e^{i\theta_1}-r_2e^{i\theta_2}|^{2}=(r_1-r_2)^2+2r_1r_2(1-\cos(\theta_1-\theta_2))\geq 2r_1^{2}(1-\cos(\theta_1-\theta_2))=|r_1e^{i\theta_1}-r_1e^{i\theta_2}|^2$ and $|r_1e^{i\theta_1}-r_2e^{i\theta_2}|\geq|r_1-r_2|$.

 The proof of the estimate for $[Dv]_{0,\alpha}$ is analogous.
\end{proof}

\begin{lem} \label{lemma1}
For every positive $d$ and $r_{max}$, with $d<r_{max}$, there
exists a constant $C(d, r_{max})$ such that 
if $d\leq R\leq r_{max}$,
$v$ is harmonic in $\Omega$, and $\zeta$ is
a cut-off function with support within $|x|<R+\frac{2}{3}d$ and equal to $1$ for $|x|\leq R+\frac{1}{3}d$, then
$$ \left\|\Delta(v\zeta)\right\|_{\infty (\mathbb{R}^2)} + [\Delta(v\zeta)]_{0,\alpha(\mathbb{R}^2)}
\leq C(d, r_{\max}) \left\|v\right\|_{L^1(\Omega)}.$$
\end {lem}

\begin{proof}
 It is clear that we can choose $\zeta$ to be such that: $|D^{k}\zeta|\leq C_kd^{-k}$ (and then $[\zeta]_{k,\alpha(\Omega')}\leq C_{k+1}d^{-k-1}R^{1-\alpha} $
 since $\zeta\in C_{c}^{\infty}(B_{R+d}(0))$). Then, using Proposition \ref{prop2} and the estimates for $\zeta$ we get: 
\begin{align*}
|\Delta(v\zeta)|&\leq 2|\nabla v \cdot \nabla \zeta| +|v\Delta \zeta|
 \leq 
Cd^{-4}\left\|v\right\|_{L^1(\Omega)}.
\end{align*}
On the other hand:
$$ [\Delta(v\zeta)]_{0,\alpha(\Omega')}\leq 2[\nabla v \cdot \nabla \zeta]_{0,\alpha(\Omega')} +[v\Delta \zeta]_{0,\alpha(\Omega')}$$
and 
\begin{align*}
[v_{,\beta} \cdot \zeta_{,\beta}]_{0,\alpha(\Omega')}&
\leq  [v_{,\beta}]_{0,\alpha(\Omega')}\left\|\zeta_{,\beta}\right\|_{\infty(\Omega')}
+[\zeta_{,\beta}]_{0,\alpha(\Omega')}\left\|v_{,\beta}\right\|_{\infty(\Omega')}
\\
[v\Delta \zeta]_{0,\alpha(\Omega')}&\leq 
[v]_{0,\alpha(\Omega')}\left\|\Delta\zeta\right\|_{\infty(\Omega')}+
[\Delta\zeta]_{0,\alpha(\Omega')}\left\|v\right\|_{\infty(\Omega')}.
\end{align*}
Hence:
$$ [\Delta(v\zeta)]_{0,\alpha(\Omega')}\leq C(d, r_{max})\left\|v\right\|_{L^1(\Omega)}.$$
Now if $x\in \Omega'$ and $y\in \mathbb{R}^2\setminus \overline{\Omega'}$, there exists $t\in (0,1)$ such that $z=tx+(1-t)y\in \partial\Omega'$, then we have
$$ |\Delta(v\zeta)(x)-\Delta(v\zeta)(y)|\leq |\Delta(v\zeta)(x)-\Delta(v\zeta)(z)|+|\Delta(v\zeta)(z)-\Delta(v\zeta)(y)|$$
$$=|\Delta(v\zeta)(x)-\Delta(v\zeta)(z)|\leq C(d, r_{max})\left\|v\right\|_{L^1(\Omega)}|x-z|^{\alpha}$$
(clearly if $x,y\in \mathbb{R}^2\setminus \overline{\Omega'}$, $|\Delta(v(x)\zeta(x))-\Delta(v(y)\zeta(y))|=0$).
The result follows by observing that $|x-z|\leq |x-y|$.
\end{proof}
 
 \subsection{Estimates near circular boundaries}
\label{se:circles}

Recall the notation $G_N$, $\Phi$ and $\phi^x(y)$ of Section \ref{se:notation_Green}.
Let $R$, $d$, $\Omega$ and $\Omega'$ be as in \eqref{eq:Omega}-\eqref{eq:Omega'}.
The following representation formula for the solution 
to the Neumann problem
can be obtained using standard arguments
(as those in \cite{DiBenedetto09,Evans10};
a complete proof can be found in \cite[Prop.\ 5.2]{CH19singular}).

\begin{proposition} \label{prop1}
Let $v$ be harmonic in $\Omega$ and $\zeta$ be a cut-off function with support within $|x|<R+\frac{2}{3}d$ and equal to $1$ for $|x|\leq R+\frac{1}{3}d$. 
Then, if $u=\zeta v$:
$$u(x)=C-\int_{\partial B_R}\frac{\partial u}{\partial \nu}G_N(x,y)dS(y)
-\int_{\Omega}\Delta u\, G_N(x,y)dy,
 \qquad
 \nu(y):= \frac{y}{R}
 \ \text{for } y\in \partial B_R.$$
\end{proposition}

\begin{proposition}   \label{prop11}
Let $R$ be any positive number, $g$ be any function in $C^{1,\alpha}(\partial B_R(0))$,
and $$u(x):=\int_{\partial B_R(0)}g(y)G_N(x,y)dS(y).$$ Then,
in $B_{R+d}(0)\setminus  \overline{B}_R(0)$:
\begin{multline*}
  \left\|D u\right\|_{\infty}
                 +
                      [Du]_{0,\alpha}
                   + 
                   \left\|D^{2} u\right\|_{\infty}
                   +
                         [D^2 u]_{0,\alpha}
                         \\ \leq 
                         C(\min\{1,R\}^{-\alpha} + \max\{1, R\}^\alpha)(
                         \|g\|_\infty + [g]_{0,\alpha} + \|g'\|_\infty + [g']_{0,\alpha}).
\end{multline*}
\end {proposition}

\begin{proof}
Using the identity $|x_1||x_1^{*}-x_2|=|x_2||x_1-x_2^{*}|$, let us first note that:
\begin{align} \label{log-reflection}
\log|y-x^{*}|=\log|y^{*}-x|+\log|y|-\log|x|.  \end{align}
This implies that for all $y\in\partial B_R(0)$
$$G_N(x,y)=-\frac{1}{\pi}\log|y-x|+\frac{1}{2\pi}\log\frac{|x|}{R}-\frac{|y|^2}{4\pi R^2}. $$
Therefore, 
\begin{align} \label{eq:SL_reg1}
u(x)=u_1(x) + \left (\frac{1}{2\pi} \log \frac{|x|}{R}\right)\int_{\partial B_R(0)} g(y) dS(y)
- \frac{1}{4\pi R^2} \int_{\partial B_R(0)} g(y) |y|^2 dS(y),
\end{align}
with
\begin{align} \label{eq:SL_reg2}
 u_1(x)&= -\frac{1}{\pi} \int_{\partial B_R(0)} g(y) \log |x-y| dS(y).
 \end{align}
 The function $u_1$ can, in turn, be written as 
 \begin{align} \label{eq:SL_reg3}
 u_1(x) &= -\frac{\log R}{\pi} \int_{\partial B_R(0)} g(y) dS(y) - \frac{R}{\pi} u_2(\frac{x}{R})
\end{align}
with
\begin{align} \label{eq:SL_reg4}
 u_2(\hat x ):= \int_{\partial B_1(0)} \hat g(y) \log|\hat x -y| dS(y),
 \qquad \hat g(y):=g(Ry)\ \text{for }y\in \partial B_1(0).
\end{align}

From Calderon-Zygmund theory it can be seen that the singular integral $u_2$ is such that
\begin{align} \label{eq:SL_reg5}
 \|Du_2\|_\infty + [Du_2]_{0,\alpha} + \|D^2 u_2\|_\infty + [D^2u_2]_{0,\alpha}
 \leq C (\|\hat g\|_\infty + [\hat g]_{0,\alpha} + \|{\hat g}'\|_\infty + [{\hat g}']_{0,\alpha})
\end{align}
in $\{\hat x: 1<|\hat x|<2\}$
(see \cite[Prop.\ 4.5]{CH19singular} for a direct proof).
The proposition then follows from \eqref{eq:SL_reg1}--\eqref{eq:SL_reg5}
and estimates for $\log|x|$ 
(recall that for the H\"older continuity, we can proceed as in Proposition \ref{prop2}).
\end{proof}

Using the arguments in Calderon-Zygmund theory (e.g.,\ as in \cite[Thm.\ 2.6.4]{Morrey66}),
but carefully looking at the dependence on $R$, it is possible to obtain the following estimate 
(see \cite[Prop.\ 5.4]{CH19singular} for a detailed proof).

\begin{proposition} \label{prop4}
For every positive $d$ and $r_{max}$, with $d<r_{max}$, there exists a constant $C(d, r_{max})$
such that if $d\leq R\leq r_{max}$, $f\in C_c^{0,\alpha}(\Omega')$, 
and $$u(y):=\int_{\mathbb{R}^2}f(y)\,G_N(x,y)dy,$$
then, in $B_{R+d}(0)\setminus  \overline{B}_R(0)$,
\begin{align*}
                 \left\|D u\right\|_{\infty}
                 +
                      [Du]_{0,\alpha}
                   + 
                   \left\|D^{2} u\right\|_{\infty}
                   +
                         [D^2 u]_{0,\alpha}
                         \leq 
                         C(d, r_{max}) ( \|f\|_\infty + [f]_{0,\alpha}). 
                        \end{align*}
\end {proposition}

\subsection{Estimates for the Neumann problem}

\begin{proposition} \label{prop12}
Let $d$ and $r_{max}$ by arbitrary positive numbers with $d<r_{max}$. 
There exists a constant $C(d,r_{max})$ such that anytime that
$E$ is a domain satisfying \eqref{eq:genericE}--\eqref{eq:d}
and 
$u$ is the unique solution to the Neumann problem \eqref{eq:defNeumann}
for some continuous datum $g$,
$$\Vert u \Vert_{L^{1}(E)}\leq C(d, r_{max})\Vert g\Vert_{\infty}.   $$
\end {proposition}

\begin{proof}
 First note that:
$$ \int_{E}|u|dy\leq |E|^{\frac{1}{2}}\Vert u\Vert_{L^{2}(E)}\leq C_P(E)|E|^{\frac{1}{2}}\Vert Du\Vert_{L^{2}(E)},$$
$C_P(E)$ being Poincar\'e's constant \eqref{def_C_P}.
Integrating by parts we get:
$$\int_{E}u\Delta u dy=\int_{\partial E}ugdS(y)-\int_{ E}|Du|^{2}dy=0.$$
Moreover:
$$ \int_{E}|Du|^2dy\leq \Vert g\Vert_{L^2(\partial E)}\Vert u\Vert_{L^2(\partial E)}.$$
Using Cauchy's inequality, we get:
$$  \Vert Du\Vert_{L^{2}(E)}\leq \frac{1}{2^{\frac{1}{2}}}\left( A\Vert g\Vert_{L^{2}(\partial E)} +\frac{\Vert u\Vert_{L^{2}(\partial E)}}{A} \right).  $$
Furthermore, using Lemma \ref{trace} and the Poincar\'e's constant, we obtain:
$$ \int_{\partial E}u^2dS=\sum_{k=0}^{n}\int_{\partial B_{r_k}(z_k)}u^2dS\leq 
C\left(\int_{B_{r_0}(z_0)\setminus B_{r_0-d}(z_0)}d^{-1}u^2+d|Du|^2dy\right)                  $$
$$ +C\left(  \sum_{k=1}^{n}\int_{B_{r_k+d}(z_k)\setminus B_{r_k}(z_k)}d^{-1}u^2+d|Du|^2dy       \right)$$
$$\leq C\left( d^{-1}\int_{E}u^2dy+\int_{E}d|Du|^2dy    \right)\leq C(d^{-1}C_P(E)^2+d)\int_{E}|Du|^2dy$$
\noindent Choosing $A=2^{\frac{1}{2}}C(d^{\frac{-1}{2}}C_P(E)+d^{\frac{1}{2}})$ we deduce that:
$$  \Vert Du\Vert_{L^{2}(E)}\leq 2^{\frac{1}{2}}A\Vert g\Vert_{L^2(\partial E)}\leq C(d^{\frac{-1}{2}}C_P(E)+d^{\frac{1}{2}})n^{\frac{1}{2}}r_0^{\frac{1}{2}}\Vert g\Vert_{\infty}.   $$
\noindent Finally, we obtain:
$$ \Vert u \Vert_{L^{1}(E)}\leq C\cdot |E|^{\frac{1}{2}}C_P(E)(d^{\frac{-1}{2}}C_P(E)+d^{\frac{1}{2}})n^{\frac{1}{2}}r_0^{\frac{1}{2}}\Vert g\Vert_{\infty}.   
$$
The proposition follows by applying Theorem \ref{thm 2}
(and the remark after its statement, which shows that $n\leq \delta^{-2}$)
with $\delta=\frac{1}{2}\frac{d}{r_{max}}$. (Note that $|E|\leq \pi r_0^2 \leq \pi r_{max}^2$.)
\end{proof}

\begin{thm}
  \label{thm1}
  Let $d$ and $r_{max}$ by arbitrary positive numbers with $d<r_{max}$. 
There exists a constant $C(d,r_{max})$ such that anytime that
$E$ is a domain satisfying \eqref{eq:genericE}--\eqref{eq:d}
and 
$u$ is the unique solution to the Neumann problem \eqref{eq:defNeumann}
for some  $g\in C^{1,\alpha}(\bigcup_{k=0}^n \partial B_{r_k}(z_k))$,
$$
  \|Du\|_\infty(E) + [Du]_{0,\alpha}(E) + \|D^2 u\|_\infty (E)
  + [D^2u]_{0,\alpha}(E)
  \leq C(d,r_{max})(\|g\|_\infty + [g]_{0,\alpha} +\|g'\|_\infty +[g']_{0,\alpha}).
  $$
\end{thm}

\begin{proof}
 Near the holes, in $\bigcup_{k=1}^n B_{r_k+d/3}(z_k)\setminus \overline{B}_{r_k}(z_k)$,
 the result follows  from Proposition \ref{prop1}, Proposition \ref{prop11}, 
Proposition \ref{prop4}, Lemma \ref{lemma1} and Proposition \ref{prop12}.

Near the exterior boundary, in $B_{r_0}(z_0)\setminus \overline{B}_{r_0-d/3}(z_0)$,
the result is obtained analogously, taking care of choosing $R=r_0-d$ in the definition 
of $\Omega$ in \eqref{eq:Omega} and of changing the representation formula
of Proposition \ref{prop1} to the corresponding one for the interior of a disk:
$$
\zeta(x) u(x)=C+\int_{\partial B_R}\frac{\partial u}{\partial \nu}\,G_N(x,y)\,dS(y)
-\int_{\Omega}\Delta (\zeta u)\, G_N(x,y)\,dy.
$$

The regularity in the interior, in $B_{r_0-d/3}(z_0)\setminus \bigcup_{k=1}^n \overline{B}_{r_k+d/3}(z_k)$,
follows from local regularity for harmonic functions and Proposition \ref{prop2} (using triangle inequality at most $2n+1$ times): 
 join $x$ and $z$ with a straight line, 
 then the segment intersects at most the $n$ holes. 
 In that case, join the points 
 using segments of the above straight line and segments of circles of the form $\partial B_{r_k+\frac{d}{3}}(z_k)$ 
 (for straight lines use local estimates for harmonic functions and for circles use Proposition \ref{prop2}).
\end{proof}

 \subsection{Estimate for the velocity fields}
  
\begin{thm}
 \label{th:Schauder}
 Let $n\in \N$ and ${\mathcal{B}}=B_{R_0}(0)\subset \R^2$. Suppose that the configuration 
  $\Big ( (a_i)_{i=1}^n, (v_i)_{i=1}^n \Big )$ is attainable. 
  Let $R_1, \ldots, R_n$ be any positive numbers 
  verifying \eqref{eq:disjointEt}, for the evolutions $r_i(t)$
defined in \eqref{eq:defRiri}.
Let $E(t)$ be as in \eqref{eq:domainsEt}.
Then the velocity fields $v(\cdot, t)$
defined in \eqref{eq:defDM}
are such that 
\begin{align*}
 &\div v(y, t)=0\quad \text{for all }t\in[1,\lambda],\ y\in E(t), \\
 & v(z_i(t) + r_i(t) e^{i\theta})= \frac{\dd r_i}{\dd t} e^{i\theta} 
  \quad \text{for all }t\in [1,\lambda],\ i\in \{1,\ldots, n\},\ \theta\in [0,2\pi],
  \\
  \text{and}\quad & \sup_{t\in [1,\lambda]} \|Dv(\cdot, t)\|_{L^\infty(E(t))}<\infty.
\end{align*}
\end{thm}

\begin{proof}
Let
$d$ be a positive number for which \eqref{eq:disjoint_d} is satisfied.
  As observed in \eqref{eq:rmax}, at all times the radii of the $n$ holes 
  are controlled by $$r_{max}:=\lambda R_0.$$
  Thus, for the maps $\psi_t$ of \eqref{varphi bvp}, Theorem \ref{thm1} yields
  $$
  \|D\varphi_t\|_\infty + \|D^2\varphi_t\|_\infty 
  \leq C(d,r_{max}) \left(
  \left\Vert \frac{\partial \phi_t}{\partial \tau} \right\Vert_{\infty} 
  +\left[ \frac{\partial \phi_t}{\partial \tau} \right]_{0,\alpha} 
  +\left\Vert\frac{\partial^2\phi_t}{\partial\tau^2}  \right\Vert_{\infty}
  +\left[ \frac{\partial^2\phi_t}{\partial\tau^2} \right]_{0,\alpha}
  \right).
  $$
  Now, it is easy to see that:
$$   \left\Vert \frac{\partial \phi_t}{\partial \tau} \right\Vert_{\infty}\leq  \Vert D\phi_t \Vert_{\infty},
\qquad   \left[ \frac{\partial \phi_t}{\partial \tau} \right]_{0,\alpha} \leq C\left(   d^{-\alpha}\Vert D\phi_t \Vert_{\infty}+[D\phi_t]_{0,\alpha}  \right)$$
(note that $\frac{\partial \phi_t}{\partial \tau}= D\phi_t \cdot \tau$, so 
the estimate of its H\"older norm needs to take into account not only spatial variations
of $D\phi_t$ but also the variation of the tangent vector from one point on $\partial B_{r_k}(z_k)$
to another).
Also, 
$$   \left\Vert \frac{\partial^2 \phi}{\partial \tau^2} \right\Vert_{\infty} \leq C \Vert D^2\phi\Vert_{\infty},
\qquad
   \left[ \frac{\partial^2 \phi}{\partial \tau^2} \right]_{0,\alpha}\leq C\left(  d^{-\alpha}\Vert D^2\phi \Vert_{\infty}+\left[D^2\phi\right]_{0,\alpha} \right)$$
since 
$$\frac{\partial^2\phi}{\partial \tau^2}=(D(D\phi\cdot \tau))\cdot \tau=(D^2\phi\cdot\tau)\cdot \tau+
\underbrace{(D\tau\cdot D\phi)\cdot \tau}_{=0}.$$
Moreover, at every fixed $t$ the function $g$ from $\partial E(t)$ to $\R$ given by 
$$g(z_i(t)+r_i(t)e^{i\theta})= \frac{\dd r_i(t)}{\dd t},\quad i\in \{1,\ldots, n\}$$
is constant on each connected component of $\partial E(t)$. Hence
Theorem \ref{thm1} gives
$$
\Vert D\phi_t\Vert_{\infty}+\left[D\phi_t\right]_{0,\alpha}+
\Vert D^2\phi_t \Vert_{\infty}+\left[D^2\phi_t\right]_{0,\alpha} \leq C(d,r_{max}) \max_i \left |\frac{\dd r_i}{\dd t}\right |.
$$
By \eqref{eq:defRiri} and item (\ref{it:Lsmooth}) in Definition \ref{de:attainable}, 
the quantity on the right-hand side, which can be written as
$$\max \frac{\left | \frac{\dd (L_i^2)}{\dd t} \right |}{r_i(t)},$$
remains bounded for all times (recall that $r_i(t)\geq d$ for all $t$). 
Consequently, $\|D^2\phi_t\|_{L^\infty}$, $\|D^2\varphi_t\|_{L^\infty}$, and 
$\|D\varphi_t\|_{L^\infty}$ remain bounded in time.

On the other hand, it is easy to see that:
$$\Vert D\psi\Vert_{\infty}\leq C\left(\frac{1}{d}\Vert\varphi\Vert_{\infty}+\Vert D\varphi\Vert_{\infty}\right),
\qquad
\Vert D^2\psi\Vert_{\infty}\leq C\left(\frac{1}{d^2}\Vert\varphi\Vert_{\infty}+\frac{1}{d}\Vert D\varphi\Vert_{\infty}+\Vert D^2\varphi\Vert_{\infty}\right).$$

Note that using the fundamental theorem of calculus one can obtain:
$\Vert \varphi \Vert_{\infty}\leq Cr_0\Vert D\varphi\Vert_{\infty}$
(using that there exists a point where $\varphi$ vanishes). 
This completes the proof.
\end{proof}

\section{Existence of deformations opening only round cavities}
\label{se:round}

In this section we consider a given attainable configuration 
$\Big ( (a_i)_{i=1}^n, (v_i)_{i=1}^n \Big )$
and work with the evolutions $z_i(t)$ and $L_i(t)$ of the centers and radii of Definition \ref{de:attainable}.
We set $\lambda>1$ to be the stretch factor given by \eqref{eq:defLambda}.
We also fix radii $R_1, \ldots, R_n$ verifying \eqref{eq:disjointEt}, for the evolutions $r_i(t)$
defined in \eqref{eq:defRiri}.

\begin{proposition} \label{pr:scission-near}
Let $n\in \N$ and ${\mathcal{B}}=B_{R_0}(0)\subset \R^2$. Suppose that the configuration 
  $\Big ( (a_i)_{i=1}^n, (v_i)_{i=1}^n \Big )$ is attainable. 
  Let $R_1, \ldots, R_n$ be any positive numbers 
  verifying \eqref{eq:disjointEt}.
 Define $u_{\text{near}}$ to be the function from 
 $\bigcup_i \overline{B}_{R_i}(a_i)$ to $\R^2$
 given by 
 $$ 
 u_{\text{near}}(a_i+re^{i\theta}) :=
  z_i(\lambda) + \sqrt{L_i(\lambda)^2 + r^2} e^{i\theta}, \quad 0<r\leq R_i,\quad i\in \{1,\ldots, n\}.
$$
Then $u_{\text{near}}$ is one-to-one a.e., satisfies $\det Du_{\text{near}}\equiv 1$ a.e.,
and is such that 
$$|\imT(u_{\text{near}}, B_\eps(a_i))|=v_i+\pi\eps^2\quad \text{for all } i\in \{1,\ldots, n\}.$$
In addition, for every small $\eps>0$
$$
    \int_{\bigcup \{x:\eps<|x-a_i|<R_i\}} \frac{|Du_{\text{near}}|^2}{2}\dd x \leq 
    \sum_{i=1}^n  \pi R_i^2 + \sum_{i=1}^n v_i\log R_i +  \left (\sum_{i=1}^n v_i\right )|\log \eps|.
$$
\end{proposition}
\begin{proof}
  Given $i\in\{1,\ldots, n\}$, $r\in (0,R_i)$ and $\theta\in[0,2\pi]$
  \begin{align}
   Du_{\text{near}}(a_i+re^{i\theta}) &= \frac{r}{\sqrt{L_i(\lambda)^2 +r^2}} e^{i\theta}\otimes e^{i\theta}
      + \sqrt{1+\frac{L_i(\lambda)^2}{r^2}} ie^{i\theta}\otimes ie^{i\theta}.
  \end{align}
  Hence $\det Du_{\text{near}}\equiv 1$ and 
  \begin{align}
   \int_{\bigcup \{x:\eps<|x-a_i|<R_i\}} \frac{|Du_{\text{near}}|^2}{2}\dd x &\leq  
   \sum_i \int_{\eps}^{R_i} \left (1+\left (1+\frac{L_i(\lambda)^2}{r^2}\right )\right ) \cdot \pi r \dd r.
  \end{align}
\end{proof}

\begin{thm}
 \label{th:scission-away}
 Let $n\in \N$ and ${\mathcal{B}}=B_{R_0}(0)\subset \R^2$. Suppose that the configuration 
  $\Big ( (a_i)_{i=1}^n, (v_i)_{i=1}^n \Big )$ is attainable. 
  Let $R_1, \ldots, R_n$ be any positive numbers 
  verifying \eqref{eq:disjointEt}.
  There
  exists $u_{\text{far}} \in H^1( \mathcal{B}\setminus \bigcup_1^n \overline{B}_{R_i}(a_i), \R^2)$, 
  satisfying $u_{\text{far}}(x)=\lambda x$ on $\partial \mathcal{B}$; $\det Du_{\text{far}}\equiv 1$ in $ 
 \mathcal{B}\setminus \bigcup_1^n \overline{B}_{R_i}(a_i)$; condition (INV); and 
 $$u_{\text{far}}(a_i+R_ie^{i\theta})= z_i(\lambda) + \sqrt{L_i(\lambda)^2+R_i^2} e^{i\theta},
 \quad \forall\,i\in\{i,\ldots, n\}\ \forall\,\theta\in [0,2\pi].$$
\end{thm}

\begin{proof}
 Let $E(t)$, $v(\cdot, t)$, and $\tilde v(\cdot, t)$ be as in \eqref{eq:domainsEt}, \eqref{eq:defDM}, and \eqref{eq:defVtilde},
 respectively.
 Observe that \begin{align}
    & \div {\tilde v}_{t}\equiv 0\text{ in } E(t)\\
    & {\tilde v}_{t}(y) = \frac{\dd z_i(t)}{\dd t}\text{ on } \partial B_{r_i(t)}(z_i(t)),
  \end{align}
  where for $i=0$ the center $z_0(t)$ is the origin and 
  the outer radius $r_0(t)$ is $tR_0$.
  As for the derivative, 
  \begin{align*}
    \|D{\tilde v}_t\|_\infty &=\max_i \Bigg \|
    \left ( d^{-2} \eta'' e^{i\theta}\otimes e^{i\theta}     
    + (dr)^{-1}\eta' ie^{i\theta}\otimes ie^{i\theta}\right ) \frac{\dd z_i(t)}{\dd t} \cdot (rie^{i\theta})
    \\ &
    \hspace{18em}
      + d^{-1}\eta' \left (  \frac{\dd z_i(t)}{\dd t}\right )^\perp \otimes e^{i\theta}
    \Bigg \|_\infty
    \\
    &\leq C(d^{-2}\cdot (\lambda R_0)+d^{-1}) \left |\frac{\dd z_i(t)}{\dd t}\right |,
  \end{align*}
  which is bounded uniformly in $t$ since $z_i\in C^1([1,\lambda],\R^2)$.
  
  For every $x\in \mathcal{B}\setminus \bigcup_1^n B_{R_i}(a_i)$ and every $t\in [1,\lambda]$ 
  let
  $f(x,t)$ be the solution of the Cauchy problem
  \begin{align}
   \begin{aligned}
     & \frac{\partial f}{\partial t} (x, t) = v_t(f(x,t)) + {\tilde v}_t(f(x,t))\\
     & f(x, 1) = x.
   \end{aligned}
  \end{align}
  It can be seen (as in Dacorogna \& Moser \cite{DaMo90})
  that the above ODE indeed
  has a well defined solution with enough regularity in time and space
  (in spite of the fact that the velocity fields are defined in changing domains).
  The  problem solved for $v$ in Theorem \ref{th:Schauder}, namely,
  \begin{align*}
    &\div v\equiv 0\ \text{in }E(t),\\
    &v(z_i(t)+r_i(t) e^{i\theta}) = \frac{\dd r_i}{\dd t}e^{i\theta},\quad i\in \{1,\ldots, n\},\ \theta\in [0, 2\pi],
  \end{align*}
  has infinite solutions and a wild (in time) selection is admissible in principle. 
  However, in the case at hand, 
  we do have a continuous dependence of the fields $v$ with respect to $t$
  since we are working with a very particular solution of the problem, namely,
  that of Dacorogna \& Moser. This solution is given, in terms of $\phi$ and $\varphi$,
  by \eqref{eq:defDM}, and both $\phi$ and $\varphi$ are uniquely determined as 
  the unique solution of two coupled well-posed Neumann problems for the Laplacian
  (where the solutions can be proved to depend continously on smooth changes of the domain $E(t)$). 

  Note also that 
  $$ f(a_i + R_i e^{i\theta}, t) = z_i(t) + r_i(t) e^{i\theta}\quad \forall\, i, \theta$$
  and $$f(R_0 e^{i\theta}, t) = tR_0 e^{i\theta}$$
  thanks to the boundary conditions for $v_t$ and ${\tilde v}_t$.
  Define $u_{\text{far}}$ by $$u_{\text{far}}(x):= f(x, \lambda), \quad x\in   \mathcal{B}\setminus \bigcup_1^n B_{R_i}(a_i).$$
  For every $i\in\{i,\ldots, n\}$ and $\theta\in [0,2\pi]$
  $$ 
  u_{\text{far}}(a_i+R_ie^{i\theta})= z_i(\lambda) + \sqrt{L_i(\lambda)^2+R_i^2} e^{i\theta}$$
  since $r_i(\lambda)= \sqrt{L_i(\lambda)^2+R_i^2}$. Also $u_{\text{far}}(x)=\lambda x$ on $\partial \mathcal{B}$.
  
  The resulting deformation $u_{\text{far}}$ is incompressible because
  \begin{align*}
    \frac{\partial}{\partial t} \det D_x f(x,t) &= \cof D_x f(x,t) \cdot D_x \frac{\partial f}{\partial t} (x,t )
    \\ &= \cof D_x f(x,t) \cdot D_x ((v_t+{\tilde v}_t)\circ f)(x,t)
    \\ &= \cof D_x f(x,t) \cdot (D_y(v_t+{\tilde v}_t)(f(x,t))D_x f(x,t))
    \\ &= (\cof D_x f(x,t)(D_x f(x,t))^T)\cdot D_y(v_t+{\tilde v}_t)(f(x,t))
    \\ &= (\det D_x f(x,t)) I\cdot D_y(v_t+{\tilde v}_t)(f(x,t))
  \end{align*}
  and the right-hand side is zero since $\div (v_t+{\tilde v}_t)\equiv 0$. 
  
  That $u_{\text{far}}$ belongs to $H^1$ can be seen from the calculations in 
  \eqref{eq:H1a}--\eqref{eq:H1b}, Theorem \ref{th:Schauder}, and 
  that $\|D\tilde v_t\|_{L^\infty(E(t)}$ is bounded in time, as already established.
  
   Finally, Ball's global invertibility theorem \cite{Ball81} shows that $u_{ext}$ is one-to-one a.e.\ 
   which combined with the previous energy estimate and \cite[Lemma 5.1]{BHM17}
   yields that $u_{\text{far}}$ satisfies condition INV.
  
\end{proof}

In the next theorem we finally state the existence of an incompressible deformation creating round cavities
of areas $v_1,\ldots, v_n$ from the cavitation points $a_1, \ldots, a_n$. 
It is an immediate consequence of Proposition \ref{pr:scission-near} and Theorem \ref{th:scission-away}.

\begin{mainthm} \label{th:round}
Let $n\in \N$, $R_0>0$, and ${\mathcal{B}}:=B_{R_0}(0)\subset \R^2$. 
Suppose that the configuration $\Big ( (a_i)_{i=1}^n, (v_i)_{i=1}^n \Big )$ is attainable.
Let $\lambda>1$ be given by $\sum v_i=(\lambda^2 -1)\pi R_0^2$.
Then there exists 
$u\in \bigcap_{1\leq p < 2} W^{1,p}({\mathcal{B}},\R^2) \cap H^1_{\loc}({\mathcal{B}}\setminus \{a_1, \ldots, a_n\},\R^2)$
satisfying
\begin{itemize}
  \item the invertibility condition (INV) of Definition \ref{de:INV};
 \item $u(x)=\lambda x$ for $x\in \partial \mathcal{B}$;
 \item $\det Du(x)=1$ for a.e.\ $x\in \mathcal B$;
  \item the cavities $\imT(u,a_i)$ (as defined in Definition \ref{de:imT}) are disks of areas $v_i$, for all $i\in \{1,\ldots, n\}$;
  \item there exists a constant $C=C\big (n,R_0, (a_i)_{i=1}^n, (v_i)_{i=1}^n \big )$ such that for all $\eps>0$
  \begin{align}
    \label{eq:energyDM}
    \int_{\mathcal B \setminus \bigcup_1^n \overline{B}_\eps (a_i)} \frac{|Du|^2}{2}dx \leq 
    C + \left ( \sum_{i=1}^n v_i \right ) |\log \eps|.
  \end{align}

\end{itemize}

\end{mainthm}

\section{Radial symmetry of the cavities opened by the minimizers}
\label{se:coalescence}

\subsection{The case of prescribed final areas}

In this section we prove that the actual energy minimizers (and not just
the test function constructed in Theorem \ref{th:round})
also opens cavities that are circular in the $\eps\to 0$ limit.
Since the result is based on Theorem \ref{th:round}, 
it holds provided 
the cavitation configuration $\Big ( (a_i)_{i=1}^n , (v_i)_{i=1}^n\Big )$
is attainable 
through an evolution of circular cavities (Definition \ref{de:attainable}).
Since a cavitation configuration is taken to be what determines the load
being exerted on the body (as explained in Section \ref{se:physical_motivation}),
we interpret the theorem as saying that the loads satisfying the conditions
in Definition \ref{de:attainable}
are not yet
large enough to trigger the coalescence of the cavities.

\begin{mainthm} \label{th:main}
  Let $n\in \N$ and ${\mathcal{B}}:=B_{R_0}(0)\subset \R^2$.
  Suppose that the configuration 
  $\Big ( (a_i)_{i=1}^n, (v_i)_{i=1}^n \Big )$ is attainable.
  Let $\lambda>1$ be given by $\sum_1^n v_i = (\lambda^2 -1)\pi R_0^2$.
  Let $\eps_j\to 0$ be a sequence that we will denote in what follows simply by $\eps$. 
  Set ${\mathcal{B}}_\eps:= {\mathcal{B}} \setminus \bigcup_{i=1}^n \overline{B}_\eps(a_i)$.
  Assume that for every $\eps$ the map $u_\eps$ minimizes $\int_{{\mathcal{B}}_\eps} |Du|^2\dd x$
  among all $u\in H^1({\mathcal{B}}_\eps;\R^2)$ satisfying
  \begin{itemize}
   \item the invertibility condition (INV) of Definition \ref{de:INV};
   \item $u(x)=\lambda x$ for $x\in \partial {\mathcal{B}}$;
   \item $\det Du(x)=1$ for a.e.\ $x\in{\mathcal{B}_\eps}$;
   \item and $|\imT(u, B_{\eps}(a_i))|=v_i +\pi \eps^2$ for all $i\in \{1,\ldots, n\}$.
  \end{itemize}
  Then there exists a subsequence (not relabelled) and 
  $u\in \bigcap_{1\leq p < 2} W^{1,p}({\mathcal{B}},\R^2) \cap H^1_{\loc}({\mathcal{B}}\setminus \{a_1, \ldots, a_n\},\R^2)$
  such that 
  \begin{itemize}
   \item $u_\eps\weakc u$ in $H^1_{\loc} ({\mathcal{B}}\setminus \{a_1, \ldots, a_n\},\R^2)$;
   \item $\Det Du_\eps \weakcs \Det Du$ in ${\mathcal{B}}\setminus \{a_1, \ldots, a_n\}$;
   locally in the sense of measures (where $\Det Du$ is the distributional Jacobian of 
   Definition \ref{de:DetDu});
   \item $\Det Du = \sum_{i=1}^n v_i \delta_{a_i} + \mathcal L^2$ in ${\mathcal{B}}$ (where $\mathcal L^2$ is
   the Lebesgue measure);
   \item The cavities $\imT(u, a_i)$ (as defined in Definition \ref{de:imT}) are disks of area $v_i$, for all $i\in \{1,\ldots,n\}$;
    \item $|\imT(u_\eps, B_{\eps}(a_i)) \triangle \imT(u,a_i)|\to 0$ as $\eps\to 0$ for $i\in\{1,\ldots,n\}$.
  \end{itemize}
\end{mainthm}

\begin{xrem}
The conclusions of the theorem are the same as those in \cite[Thm.\ 1.9]{HeSe13}; 
the reason is that the former is obtained by applying the latter.
 What differs is that the conclusions are obtained under a different set of hypotheses.

 The main assumption  in \cite[Thm.\ 1.9]{HeSe13} is that a
constant $C$ (independent of $\eps$) exists such that 
  \begin{align} \label{eq:uBound}
      \int_{{\mathcal{B}}_\eps} \frac{|Du_\eps|^2}{2}\dd x \leq 
  C + \left ( \sum_{i=1}^n v_i \right ) |\log \eps|.
  \end{align}
  Recall that the cost of opening round cavities of areas $v_1, \ldots, v_n$
  is $(\sum v_i) |\log \eps|$.
  In a sense, 
  this is to be expected since the singularity in the gradient of a map creating a cavity from a single point $a\in {\mathcal{B}}$
is at least of the order of 
  \begin{align*}
    |Du(x)|\sim \frac{L}{r},
    \quad
    r=|x-a|
\end{align*}
where $L$ is such that $ \pi L^2$ equals  the area of the created cavity. 
In light of \eqref{eq:ameliore}, condition \eqref{eq:uBound}
yields that all the distortions are zero, hence all cavities
are round, as stated in the theorem.
From \eqref{eq:ameliore} we see that 
leaving the space of deformations that open only round cavities
comes with an energetic cost of order $|\log \eps|$ (in addition to the
$\sum_i v_i |\log \eps|$ common to all maps in the admissible space).
Therefore,
the elongation and coalescence of voids corresponds to \emph{a 
higher energy regime};
condition \eqref{eq:uBound}, 
in contrast, 
characterizes the lowest energy regime 
where the Dirichlet enegy blows up at no more that
the stated rate of $\sum_i v_i |\log \eps|$,
which corresponds to loads not large enough so as to 
initiate the merging of cavities. 

The natural question arising from \cite[Thm.\ 1.9]{HeSe13}
is what load configurations lie in the lowest energy regime 
\eqref{eq:uBound}.
As mentioned in Section \ref{se:Int_DM},
Henao \& Serfaty \cite{HeSe13} answered this for the case of two cavities,
using explicit constructions of incompressible maps opening cavities
of all possible sizes from a pair of arbitrary cavitation points.
The novelty on this work is that we now solve the nonlinear equation of incompressibility 
for an arbitrarily large number of cavities, using instead 
the flow of Dacorogna \& Moser \cite{DaMo90} (see Section \ref{se:round}).
\end{xrem}

\begin{proof}[Proof of Theorem \ref{th:main}]
  By Theorem \ref{th:round} there exists $u$, defined in all of $\mathcal B \setminus \{a_1, \ldots, a_n\}$,
  which is a radially symmetric cavitation in a neighbourhood of each $a_i$. 
  For each $\eps>0$ let ${\tilde u}_\eps$ denote the restriction of $u$ to $\mathcal B_\eps$.
  Since the sequence $({\tilde u}_\eps)_\eps$ clearly fulfils \eqref{eq:uBound},
  and since $\int |D u_\eps| \leq \int |D {\tilde u}_\eps|$ (because, by hypothesis, the $u_\eps$ are
  energy minimizers), the sequence $(u_\eps)_\eps$ also satisfies \eqref{eq:uBound}.
  The result then follows by applying the arguments \cite[Thm. 1.9]{HeSe13}.
\end{proof}

\subsection{Lower bound for the coalescence load}

  As mentioned in p.\ \pageref{obstruction}, the problem of interest is 
to understand how cavities will continue to evolve once they have attained a certain 
size of order 1 (that is, a size much larger than the one at the rest state -or at the onset of fracture).
Motivated by this more realistic problem,
in this section we consider a constraint of the form 
\begin{align} \label{eq:upsilon}
  v_i\geq \upsilon_i \quad \forall i \in \{1,\ldots, n\}
\end{align}
for the final areas 
$$v_i:= \lim_{\eps\to 0} |\imT(u, B_\eps(a_i))|$$
of the cavities, with minimum areas $\upsilon_i>0$ that are specified a priori.
Note that this is  important also in light of Proposition \ref{pr:n=1}: if 
it is possible for all except one of the $v_i$ to be equal to zero,
then nothing regarding void coalescence can be deduced from our analysis.
Theorem \ref{th:main} treats the problem of opening cavities of prespecified areas $v_1, \ldots, v_n$.
Here, in contrast, the Dirichlet energy will be minimized 
in the space $\mathcal A_\eps$ of maps  
$u\in H^1 (\mathcal B_\eps; \R^2)$ satisfying
\begin{align} \label{eq:Aeps}
    \begin{aligned}
     \bullet\  & \text{the invertibility condition (INV) of Definition \ref{de:INV};}
     \\
     \bullet\  & u(x)=\lambda x\ \text{for}\ x\in \partial \mathcal B;
     \\
     \bullet\  & \det D u(x)=1\ \text{for a.e.}\ x\in \mathcal B_\eps;
     \\
     \bullet\  & |\imT(u, B_\eps(a_i))| \geq \upsilon_i + \pi \eps^2\ \text{for all}\ i\in \{1,\ldots, n\}.
    \end{aligned}
 \end{align}
For these variational problems the result can finally be phrased in terms only 
of the displacement of the outer boundary.

\begin{mainthm} \label{pr:example_attainable}
      Let $n\in \N$, $\mathcal B=B_{R_0}(0)\subset \R^2$, and $a_1, \ldots, a_n\in \mathcal{B}$ be given. 
    Let $\overline{B}_{d_1}(a_1)$, \ldots, $\overline{B}_{d_n}(a_n)$ be a disjoint collection 
    of closed balls contained in $\mathcal B$.
    Let $\sigma^*:=\frac{\sum_k \pi d_k^2}{\pi R_0^2}$ denote its associated packing density.
    Let $\upsilon_1,\ldots, \upsilon_n>0$ be given and suppose that
    \begin{align} \label{eq:upsilonZ}
     \upsilon_i < \pi d_i^2 \cdot \frac{1}{1-\sigma^*} \text{ for each } i\in \{1,\ldots, n\}.
    \end{align}
    Then there exists $\lambda_0\in (1, \frac{1}{\sqrt{1-\sigma^*}})$ such that given any 
    \begin{align} \label{eq:criticalLoad}
     \lambda_0\leq \lambda <\frac{1}{\sqrt{1-\sigma^*}};
    \end{align}
    any sequence $\eps_j\to 0$ (which we denote in what follows simply by $\eps$);
    and any sequence $(u_\eps)_\eps$ of minimizers of $\int_{\mathcal B_\eps} |Du|^2 dx$
    in the spaces $\mathcal A_\eps$ of \eqref{eq:Aeps};
    all of the conclusions of Theorem \ref{th:main} hold
    (in particular, the maps $u_\eps$ tend to produce only round cavities in the limit as $\eps\to 0$).
\end{mainthm}

\begin{remarks}
  \begin{enumerate}
   \item Based on the discussions of this section, we interpret the value 
  of $\lambda$ found in \eqref{eq:criticalLoad},
  namely, $\left (\displaystyle 1-\frac{\sum_k \pi d_k^2}{\pi R_0^2}\right )^{-\frac{1}{2}}$,
  as a lower bound for the coalescence load for this problem. 
  \item Thinking of a quasistatic loading, the theorem says that even if $n$ cavities have already formed
  and grown inside the body, it is still possible to continue loading it without entering the stage of void coalescence
  provided that their current radii $\sqrt{\frac{\upsilon_i}{\pi}}$ are less than $\frac{d_i}{\sqrt{1-\sigma^*}}$. 
  As mentioned at the end of Section \ref{se:Int_coalescence}, this suggests that if even one of the cavities has not yet attained
  that characteristic size then no coalescence should be expected (because that cavity still has room to  grow
  as a round cavity, sustaining itself the global effect of the increment in the external load).
  \item Observe that $\frac{1}{1-\sigma^*}\to\infty$ as $\sigma^*\to 1^-$. This has an effect 
  both on the coalescence load (which is larger than $\sqrt{\frac{1}{1-\sigma^*}}$) and
  on the critical final radius $\frac{d_i}{\sqrt{1-\sigma^*}}$ for a circular cavity.
  This suggests that
  the energetically most favourable situation is when the space available in the reference configuration
  $B_{R_0}(0)$ is optimally distributed among all the balls $B_{d_i}(a_i)$. This occurs either
  when the body opens only one cavity, or  at the other end when the body opens 
  a larger and larger number of smaller cavities. The second possibility is more realistic, due to 
  the dynamic and irreversible nature of the fracture processes and
  due to local vs.\ global minimization considerations. 
  What prevents 
  an arbitrarily large number of cavities from being created are the energies required for fracture
  (see \cite{Mora14}) and the tension associated to the presence of an ever increasing inner surface
  (which is especially large since what matters is its state in 
  the deformed configuration, as pointed out by M\"uller \& Spector \cite{MuSp95}).
  \end{enumerate}  
\end{remarks}

\begin{proof}
  We proceed as in the proof of Theorem \ref{th:main}, except that now 
  $v_1,\ldots, v_n$ are to be found such that $v_i\geq \upsilon_i$ for all $i$
  and the configuration   $\Big ( (a_i)_{i=1}^n, (v_i)_{i=1}^n \Big )$ is attainable. 
  Choose the $v_1, \ldots, v_n$ given by \eqref{eq:proportional2}.
  Thanks to \eqref{eq:max_lambda}, $\lambda<\frac{1}{\sqrt{1-\sigma^*}}$ is enough to ensure
  that the configuration is attainable. 
  From \eqref{eq:proportional2} we also see that $v_i\geq \upsilon_i$ if and only if
  $\lambda^2 \geq 1+ \frac{\upsilon_i}{\pi d_i^2} \sigma^*$.
  This holds for each $i$ if and only if 
  $$ \lambda \geq  \lambda_0:= \sqrt{1 + \left ( \max_i \frac{\upsilon_i}{\pi d_i^2} \right ) \sigma^*}.$$
  Note, in turn, that $\lambda_0 < \frac{1}{\sqrt{1-\sigma^*}}$
  (which is necessary for \eqref{eq:criticalLoad} to be meaningful)
  if and only if \eqref{eq:upsilonZ} is satisfied. 
  The conclusion now follows 
  by applying the arguments in \cite[Thm.\ 1.9]{HeSe13}; 
  the hypothesis on the blow-up rate of the energy (as $\eps\to 0$) 
  is satisfied thanks to Theorem \ref{th:round}.
\end{proof}

\subsection*{Acknowledgements}

We are indebted to Sergio Conti, Mat\'ias Courdurier, Manuel del Pino,
Robert Kohn, Giuseppe Mingione, 
Tai Nguyen and Sylvia Serfaty for our discussions 
and their suggestions. 
We are also grateful to the anonymous referees for helping us to significantly 
improve the presentation.
This research was supported by 
the FONDECYT projects 1150038 and 1190038 of the Chilean Ministry of Education 
and by the
Millennium Nucleus Center for Analysis of PDE NC130017 
of the Chilean Ministry of Economy.
D.H.\ thanks the City Library of San Fernando, VI Regi\'on, where part of this work was carried out, 
for their warm hospitality.

\bibliography{biblio} \bibliographystyle{alpha}

\begin{thebibliography}{PLLPRC17}

\bibitem[Bal81]{Ball81}
J~M Ball.
\newblock {Global invertibility of Sobolev functions and the interpenetration
  of matter}.
\newblock {\em Proc. Roy. Soc. Edinb. Sect. A}, 88(3-4):315--328, 1981.

\bibitem[Bal82]{Ball82}
J~M Ball.
\newblock {Discontinuous equilibrium solutions and cavitation in nonlinear
  elasticity}.
\newblock {\em Philos. Trans. R. Soc. Lond. Ser. A}, 306:557--611, 1982.

\bibitem[BFM08]{BoFrMa08}
Blaise Bourdin, Gilles~A Francfort, and Jean-Jacques Marigo.
\newblock {The variational approach to fracture}.
\newblock {\em J. Elasticity}, 91(1-3):5--148, 2008.

\bibitem[BHMC17]{BHM17}
Marco Barchiesi, Duvan Henao, and Carlos Mora-Corral.
\newblock {Local invertibility in Sobolev spaces with applications to nematic
  elastomers and magnetoelasticity}.
\newblock {\em Arch. Rational Mech. Anal.}, 224(2):743--816, 2017.

\bibitem[BM84]{BaMu84}
J~M Ball and F~Murat.
\newblock {{\{}{\$}W{\^{}} {\{}1,p{\}}{\$}{\}}-quasiconvexity and variational
  problems for multiple integrals}.
\newblock {\em J. Funct. Anal.}, 58(3):225--253, 1984.

\bibitem[CAH]{CH19singular}
V.~Ca{\~{n}}ulef-Aguilar and D.~Henao.
\newblock H{\"{o}}lder estimates for the {N}eumann problem in a domain with
  holes and a relation formula between the {D}irichlet and {N}eumann problems.
\newblock Preprint available at \url{mat.uc.cl/prepublicaciones}.

\bibitem[DiB09]{DiBenedetto09}
E~DiBenedetto.
\newblock {\em {Partial Differential Equations}}.
\newblock Birkhauser, 2009.

\bibitem[DM90]{DaMo90}
B~Dacorogna and J~Moser.
\newblock {On a partial differential equation involving the Jacobian
  determinant}.
\newblock {\em Ann. Inst. H. Poincar{\'{e}} Anal. Non Lin{\'{e}}aire},
  7(1):1--26, 1990.

\bibitem[Eva10]{Evans10}
Lawrence~C Evans.
\newblock {\em {Partial differential equations}}, volume~19 of {\em Graduate
  Studies in Mathematics}.
\newblock American Mathematical Society, Providence, RI, second edition, 2010.

\bibitem[FGLP]{Francfort18}
G.A. Francfort, A.~Giacomini, and O.~Lopez-Pamies.
\newblock A first step towards a variational view of cavitation.
\newblock Under revision.

\bibitem[FMP08]{FuMaPr08}
N~Fusco, F~Maggi, and A~Pratelli.
\newblock {The sharp quantitative isoperimetric inequality}.
\newblock {\em Ann. Math. (2)}, 168:941--980, 2008.

\bibitem[Gen91]{Gent91}
A~N Gent.
\newblock {Cavitation in rubber: a cautionary tale}.
\newblock {\em Rubber Chem. Tech.}, 63:G49--G53, 1991.

\bibitem[GL59]{GeLi59}
A~N Gent and P~B Lindley.
\newblock {Internal rupture of bonded rubber cylinders in tension}.
\newblock {\em Proc. Roy. Soc. London Ser. A}, 249:195--205, 1959.

\bibitem[HMC11]{HeMo11}
Duvan Henao and Carlos Mora-Corral.
\newblock {Fracture surfaces and the regularity of inverses for BV
  deformations}.
\newblock {\em Arch. Rational Mech. Anal.}, 201:575--629, 2011.

\bibitem[HMC12]{HeMo12}
Duvan Henao and Carlos Mora-Corral.
\newblock {Lusin's condition and the distributional determinant for
  deformations with finite energy}.
\newblock {\em Adv. Calc. Var.}, 5:355--409, 2012.

\bibitem[HMC15]{HeMo15}
Duvan Henao and Carlos Mora-Corral.
\newblock {Regularity of inverses of Sobolev deformations with finite surface
  energy}.
\newblock {\em J. Funct. Anal.}, 268:2356--2378, 2015.

\bibitem[HMCX15]{HeMoXu15}
Duvan Henao, Carlos Mora-Corral, and Xianmin Xu.
\newblock {Gamma-convergence approximation of fracture and cavitation in
  nonlinear elasticity}.
\newblock {\em Arch. Rational Mech. Anal.}, 216(3):813--879, 2015.

\bibitem[HMCX16]{HeMoXu16}
Duvan Henao, Carlos Mora-Corral, and Xianmin Xu.
\newblock {A numerical study of void coalescence and fracture in nonlinear
  elasticity}.
\newblock {\em Comput. Methods Appl. Mech. Engrg.}, 303:163--184, 2016.

\bibitem[HS13]{HeSe13}
Duvan Henao and Sylvia Serfaty.
\newblock {Energy Estimates and Cavity Interaction for a Critical-Exponent
  Cavitation Model}.
\newblock {\em Comm. Pure Appl. Math.}, 66:1028--1101, 2013.

\bibitem[KFLP18]{Kumar18}
A.~Kumar, G.A. Francfort, and O.~Lopez-Pamies.
\newblock Fracture and healing of elastomers: A phase-transition theory and
  numerical implementation.
\newblock {\em J. Mech. Phys. Solids}, 112:523--551, 2018.

\bibitem[KRCLP]{Kumar18b}
A.~Kumar, K.~Ravi-Chandar, and O.~Lopez-Pamies.
\newblock The configurational-forces view of fracture and healing in elastomers
  as a phase transition.
\newblock Under revision.

\bibitem[LL11a]{LianLi11JCPAM}
Yijiang Lian and Zhiping Li.
\newblock {A dual-parametric finite element method for cavitation in nonlinear
  elasticity}.
\newblock {\em J. Comput. Appl. Math.}, 236(5):834--842, 2011.

\bibitem[LL11b]{LianLi11}
Yijiang Lian and Zhiping Li.
\newblock {A numerical study on cavitation in nonlinear elasticity---defects
  and configurational forces}.
\newblock {\em Math. Models Methods Appl. Sci.}, 21(12):2551--2574, 2011.

\bibitem[LL12]{LianLi12}
Yijiang Lian and Zhiping Li.
\newblock {Position and size effects on voids growth in nonlinear elasticity}.
\newblock {\em Int. J. Fract.}, 173(2):147--161, 2012.

\bibitem[LPIN11]{LoIdNa11}
O~Lopez-Pamies, M~I Idiart, and T~Nakamura.
\newblock {Cavitation in elastomeric solids: I — A defect-growth theory}.
\newblock {\em J. Mech. Phys. Solids}, 59:1464--1487, 2011.

\bibitem[LRCLP15]{Lefevre15}
V~Lefevre, K~Ravi-Chandar, and O~Lopez-Pamies.
\newblock {Cavitation in rubber: an elastic instability or a fracture
  phenomenon}.
\newblock {\em Int. J. Fract.}, 192:1--23, 2015.

\bibitem[MC14]{Mora14}
Carlos Mora-Corral.
\newblock {Quasistatic evolution of cavities in nonlinear elasticity}.
\newblock {\em SIAM J. Math. Anal.}, 46(1):532--571, 2014.

\bibitem[Mel94]{Melissen94}
H~Melissen.
\newblock {Densest packing of eleven congruent circles in a circle}.
\newblock {\em Geometriae Dedicata}, 50:15--25, 1994.

\bibitem[{Mor}66]{Morrey66}
Charles~B {Morrey Jr.}
\newblock {\em {Multiple integrals in the calculus of variations}}.
\newblock Die Grundlehren der mathematischen Wissenschaften, 130. Springer, New
  York, 1966.

\bibitem[MS95]{MuSp95}
Stefan M{\"{u}}ller and Scott~J Spector.
\newblock {An existence theory for nonlinear elasticity that allows for
  cavitation}.
\newblock {\em Arch. Rational Mech. Anal.}, 131(1):1--66, 1995.

\bibitem[NMS12]{NeSi12}
Pablo~V Negr{\'{o}}n-Marrero and Jeyabal Sivaloganathan.
\newblock {A characterisation of the boundary displacements which induce
  cavitation in an elastic body}.
\newblock {\em J. Elasticity}, 109(1):1--33, 2012.

\bibitem[PCSE06]{PeCuSiEl06}
N~Petrinic, J.~L. {Curiel Sosa}, C.~R. Siviour, and B.~C.~F. Elliott.
\newblock {Improved predictive modelling of strain localisation and ductile
  fracture in a {\{}Ti-6Al-4V{\}} alloy subjected to impact loading}.
\newblock {\em Journal de Physique IV}, 134:147--155, 2006.

\bibitem[PLLPRC17]{Poulain17}
X.~Poulain, V.~Lef{\`{e}}vre, O.~Lopez-Pamies, and K.~Ravi-Chandar.
\newblock {Damage in elastomers: Nucleation and growth of cavities,
  micro-cracks, and macro-cracks}.
\newblock {\em Int. J. Fract.}, 205:1--21, 2017.

\bibitem[PLPRC]{Poulain18}
X.~Poulain, O.~Lopez-Pamies, and K.~Ravi-Chandar.
\newblock Damage in elastomers: Healing of internally nucleated cavities and
  micro-cracks.
\newblock Under revision.

\bibitem[SS00]{SiSp00}
J~Sivaloganathan and S~J Spector.
\newblock {On the existence of minimizers with prescribed singular points in
  nonlinear elasticity}.
\newblock {\em J. Elast.}, 59, 2000.

\bibitem[SS10a]{SiSp10a}
J~Sivaloganathan and S~J Spector.
\newblock {On the symmetry of energy-minimising deformations in nonlinear
  elasticity {I}: incompressible materials}.
\newblock {\em Arch. Rational Mech. Anal.}, 196:363--394, 2010.

\bibitem[SS10b]{SiSp10b}
J~Sivaloganathan and S~J Spector.
\newblock {On the symmetry of energy-minimising deformations in nonlinear
  elasticity {II}: compressible materials}.
\newblock {\em Arch. Rational Mech. Anal.}, 196:395--431, 2010.

\bibitem[XH11]{XuHe11}
Xianmin Xu and Duvan Henao.
\newblock {An efficient numerical method for cavitation in nonlinear
  elasticity}.
\newblock {\em Math. Models Methods Appl. Sci.}, 21:1733--1760, 2011.

\end{thebibliography}

\end{document}